\documentclass{amsart}
\usepackage{amssymb, latexsym, euscript}
\usepackage[usenames,dvipsnames,svgnames,table]{xcolor}

   \def\sH{{\mathfrak H}}   
   \def\sK{{\mathfrak K}}   
   \def\sN{{\mathfrak N}}

\def\st{{\mathfrak t}}

\def\ss{{\mathfrak s}}

\def\sr{{\mathfrak r}}

      \def\dC{{\mathbb C}}

      \def\dR{{\mathbb R}}

      \def\cR{{\mathcal R}}
      
   \def\cW{{\mathcal W}}

\def\bB{\mathbf{B}}
\def\bL{\mathbf{L}}
\def\bF{\mathbf{F}}
\def\bLR{{\mathbf{L}}}

\def\wt#1{{{\widetilde #1} }}

\def\wh#1{{{\widehat #1} }}

\def\bm\chi{\mbox{\boldmath$\chi$}}
\def\half{{\frac{1}{2}}}

\def\ker{{\rm ker\,}}
\def\cker{{\rm \overline{ker}\,}}
\def\ran{{\rm ran\,}}
\def\cran{{\rm \overline{ran}\,}}
\def\dom{{\rm dom\,}}
\def\mul{{\rm mul\,}}
\def\cmul{{\rm \overline{mul}\,}}
\def\cdom{{\rm \overline{dom}\,}}
\def\clos{{\rm clos\,}}

\let\xker=\ker \def\ker{{\xker\,}}

\def\uphar{{\upharpoonright\,}}

\DeclareMathOperator{\hplus}{\, \widehat + \,}

\newtheorem{theorem}{Theorem}[section]
\newtheorem{proposition}[theorem]{Proposition}
\newtheorem{corollary}[theorem]{Corollary}
\newtheorem{lemma}[theorem]{Lemma}
\newtheorem{definition}[theorem]{Definition}

\theoremstyle{definition}
\newtheorem{example}[theorem]{Example}
\newtheorem{remark}[theorem]{Remark}

\numberwithin{equation}{section}

\begin{document}

\title[Friedrichs and Kre\u{\i}n type extensions]
{Friedrichs and Kre\u{\i}n type extensions\\ in terms of representing maps}
\author[S.~Hassi]{S.~Hassi}
\author[H.S.V.~de~Snoo]{H.S.V.~de~Snoo}

\address{Department of Mathematics and Statistics \\
University of Vaasa \\
P.O. Box 700, 65101 Vaasa \\
Finland} \email{sha@uwasa.fi}

\address{Johann Bernoulli Institute for Mathematics and Computer Science\\
University of Groningen \\
P.O. Box 407, 9700 AK Groningen \\
Nederland}
\email{h.s.v.de.snoo@rug.nl}

\thanks{}
%
%

\dedicatory{Dedicated to Professor~Fedor Sukochev}

\date{\today}

%

\begin{abstract}
A semibounded operator or relation $S$ in a Hilbert space
with lower bound $m \in \dR$ has a symmetric extension
$S_{\rm f}=S \hplus (\{0\} \times \mul S^*)$,
the weak Friedrichs extension of $S$,
and a selfadjoint extension $S_{\rm F}$, the Friedrichs extension of $S$,
that satisfy $S \subset S_{\rm f} \subset S_{\rm F}$.
The Friedrichs extension $S_{\rm F}$ has lower bound $\gamma$
and it is the largest semibounded  selfadjoint extension of $S$.
Likewise, for each $c \leq \gamma$, the relation $S$ has a
weak Kre\u{\i}n type extension $S_{{\rm k},c}=S \hplus (\ker (S^*-c) \times \{0\})$
and Kre\u{\i}n type extension $S_{{\rm K},c}$ of $S$,
that satisfy $S \subset S_{{\rm k},c} \subset S_{{\rm K},c}$.
The Kre\u{\i}n type extension $S_{{\rm K},c}$ has lower bound $c$
and it is the smallest semibounded selfadjoint
extension of $S$ which is bounded below by $c$.
In this paper these special extensions and, more generally, all extremal
extensions of $S$ are constructed in terms of
a representing map for $\st(S)-c$ and their properties are being considered.
\end{abstract}

\maketitle

\section{Introduction}

Let $S \in \bL(\sH)$ be linear operator or relation in a Hilbert space $\sH$
which has lower bound $m(S)=\gamma \in \dR$, i.e.,  $\gamma$
is the supremum of all $c \in \dR$ for which
\begin{equation}\label{semi}
(\varphi', \varphi) \geq c (\varphi, \varphi),  \quad \{\varphi, \varphi' \} \in S.
\end{equation}
In what follows the notation $\bL(\sH, \sK)$ stands for the class of all linear relations
from $\sH$ to the Hilbert space $\sK$; this is abbreviated to $\bL(\sH)$ when $\sK=\sH$.
It is clear from \eqref{semi} that a semibounded relation is automatically symmetric.
Among all semibounded selfadjoint extensions $H \in \bL(\sH)$ of $S$ there exists a largest
semibounded selfadjoint extension $S_{\rm F} \in \bL(\sH)$, the Friedrichs extension,
with the same lower bound as $S$,
such that $H \leq S_{\rm F}$.
Likewise, for any choice of $c \leq \gamma$, there exists
among all semibounded selfadjoint extensions $H \in \bL(\sH)$ of $S$
a smallest semibounded selfadjoint extension $S_{{\rm K},c} \in \bL(\sH)$
the so-called Kre\u{\i}n type extension, with lower bound $c$,
such that for any selfadjoint extension $H \in \bL(\sH)$ of $S$ with $H \geq c$
one has  $H \geq S_{{\rm K},c}$.
Thus, for any semibounded selfadjoint extension
$H$ of $S$ one has the equivalence
\begin{equation}\label{coddin0}
 c \leq H \quad \Leftrightarrow \quad S_{{\rm K},c} \leq H \leq S_{\rm F}.
\end{equation}
For the ordering of semibounded selfadjoint operators
and relations as in \eqref{coddin0}, see \cite{BHS}.
For each $c \leq \gamma$ the Friedrichs extension $S_{\rm F} \in \bL(\sH)$
and the Kre\u{\i}n type extension $S_{{\rm K},c} \in \bL(\sH)$ are connected
as linear relations by the identity
\begin{equation}\label{coddin}
 S_{{\rm K},c}= c+(((S-c)^{-1})_{\rm F})^{-1}.
\end{equation}
In addition, there are also weak versions of the Friedrichs
and Kre\u{\i}n type extensions, which are connected as in \eqref{coddin},
and they play an important role in applications.
Note that for $\gamma \geq 0$ and $c=0$ these results (in the context of
densely defined linear operators) go back to
J. von Neumann \cite{Neu} and M.G. Kre\u{\i}n \cite{Kr1, Kr2}; and then
one usually speaks of the Kre\u{\i}n-von Neumann extension
instead of the Kre\u{\i}n type extension.
For a recent treatment of all the above extensions
in the context of linear relations, see \cite{BHS},
but one may also consult  \cite{HMS} for various other approaches.

\medskip

The purpose of the present paper is to view the construction
of the above semibounded selfadjoint extensions
systematically from the point of view of semibounded sesquilinear forms.
In what follows $\bF(\sH)$ stands for the set of all sesquilinear forms
whose domains belong to the Hilbert space $\sH$.
Let $\st \in \bF(\sH)$ be any semibounded form with lower bound $\gamma \in \dR$.
For each $c \leq \gamma$ there is a linear operator $Q_c \in \bL(\sH, \sK)$,
where $\sK$ is a Hilbert space, the so-called representing map, such that
the nonnegative form $\st-c$ can be written as
\[
 (\st-c)[\varphi, \psi]=(Q_c \varphi, Q_c \psi), \quad \varphi, \psi \in \dom \st(S)=\dom Q_c.
\]
The specific choice of $Q_c$ does not play a central role in studying the form $\st$.
The properties of representing maps for general nonnegative forms have been
studied in detail in \cite{HS2023a, HS2023b}; see also \cite{Szy87}.
Now return to the semibounded linear relation
$S \in \bL(\sH)$ and assume that it has lower bound $\gamma \in \dR$.
Then there is a natural well-defined
sesquilinear form $\st(S) \in \bF(\sH)$,
defined by the left-hand side of \eqref{semi}, with $\dom \st(S)=\dom S$
and such that $\st(S)$ has the same lower bound as $S$.  In general,
this form $\st(S)$ is not closed; see, for instance, in \cite{C}, \cite{DS2}.
By using the specific properties of the representing map for
the form $\st(S)-c$, $c \leq \gamma$, it follows that there exists
another semibounded form $\ss(S) \in \bF(\sH)$ which is closed with
lower bound $c$, so that
\[
 \st(S) \subset \ss(S).
\]
In particular, this shows that the original form $\st(S)$ is closable.
The present paper gives an overview of the construction of
the semibounded selfadjoint extensions of the semibounded relation $S \in \bL(\sH)$
by means of appropriate semibounded forms.
The Friedrichs extension is constructed via the form $\st(S)$,
the Kre\u{\i}n type extensions via the form $\ss(S)$,
and the closely related extremal extensions via intermediate forms.
All these constructions are based on the representing map of $\st(S)-c$.
The arguments are much influenced by the work of Sebesty\'en and Stochel;
see \cite{PS, SS, SSi, ST2021, ST2022}.  It should be stressed that the treatment
in this paper is general in the sense that the possibly multivalued operator $S$
need not be closed or densely defined.

 \medskip

The contents of the paper are now described.
The forms $\st(S)$ and $\ss(S)$, induced by the semibounded relation $S \in \bL(\sH)$,
are constructed in Section \ref{sect2}, where also their relevant properties are discussed.
In Section \ref{symm} there is a further discussion of the forms $\st(S)$ and $\ss(S)$;
in particular, some attention will be paid to the (non)uniqueness of their representing maps.
Here also the connection with the construction of Sebesty\'en and Stochel
is discussed; see  Remark \ref{sebsto}.
The Friedrichs extension  $S_{\rm F}$ of $S$ and its weak version are treated
in Section \ref{Fried}.
A similar reasoning involving the corresponding inverse $(S-c)^{-1}$
with $c \leq \gamma$ is given Section \ref{dual}.
This leads to analogous extensions of $(S-c)^{-1}$ in terms of
appropriate representing maps and companion relations.
The Kre\u{\i}n type extension $S_{{\rm K},c}$ and its weak version
are introduced in Section \ref{Kreintype}. The identity in \eqref{coddin}
is a straightforward consequence of the approach
via the representing maps for the forms associated with $S$ and $(S-c)^{-1}$;
cf. Section \ref{dual}.
As an application of the above constructions,
also the notion of extremal extensions is treated in this context,
see Section \ref{eextr}. By way of illustration  Section \ref{ortho}
is devoted to symmetric relations whose domain and range are
orthogonal; see \cite{HLS, RC}.
In Section \ref{converse} attention is paid to the representation in terms of a
closable representing map of a symmetric relation generated by a
semibounded form which itself is not necessarily closed.

\section{Semibounded relations and representing maps}\label{sect2}

In this section representing maps are introduced for semibounded forms
which are generated by semibounded relations.
Also a couple of their key properties are derived for later purposes.

\begin{definition}\label{defsss}
Let $S \in \bL(\sH)$  be a semibounded relation
with lower bound $\gamma \in \dR$.
Then the semibounded sesquilinear form $\st(S) \in \bF(\sH)$
is defined by
\begin{equation}\label{s-ip}
 \st(S)[\varphi , \psi] =   (\varphi' , \psi)=(\varphi, \psi'),
 \quad \{\varphi,\varphi'\}, \, \{\psi,\psi'\} \in S,
\end{equation}
so that $\dom \st(S)=\dom S$.
\end{definition}

Since $\mul S \subset \mul S^{*}=(\dom S)^{\perp}$
the form $\st(S)$ is well-defined and, in fact,
it is semibounded with lower bound $\gamma \in \dR$.

\medskip

The new ingredient in the present treatment is the notion of the representing map.
Let $c \leq \gamma$, then for the form $\st(S)-c$
there exists a so-called \textit{representing map} $Q_c$ such that
\begin{equation}\label{forrep}
\st(S)[\varphi, \psi]=c(\varphi, \psi)+(Q_c \varphi, Q_c \psi),
\quad \varphi, \psi \in \dom \st(S)=\dom S=\dom Q_c,
\end{equation}
where $Q_c$ is a linear operator in $\bL(\sH, \sK_c)$ with a Hilbert space $\sK_c$;
see \cite[Lemma 4.1]{HS2023seq}.
Due to the definition of $\st(S)$ in \eqref{s-ip} the
representing map $Q_c$ has a so-called
 \textit{companion relation}  $J_c \in \bL(\sK_c, \sH)$, which is introduced by
\begin{equation}\label{s-J}
J_c=\big\{\, \{ Q_c \varphi, \varphi'-c \varphi\} :\, \{\varphi, \varphi'\} \in S  \,\big\}.
\end{equation}
The operator $Q_c$ and the relation $J_c$ form a dual pair in the sense that
\begin{equation}\label{s-JQ}
 Q_c \subset J_c^{\ast} \quad \mbox{or, equivalently,}
 \quad J_c \subset Q_c^*,
\end{equation}
as follows from a combination of the identities \eqref{s-ip} and \eqref{forrep}:
\begin{equation}\label{nuk}
(\varphi' -c \varphi, \psi)=(\varphi, \psi'-c \psi)=(Q_c \varphi, Q_c \psi),
\quad \{\varphi, \varphi'\}, \{\psi, \psi'\} \in S.
\end{equation}
The identity \eqref{s-J} shows that the relation $c+J_cQ_c$ is an extension of $S$:
\begin{equation}\label{s-JQQ}
 S \subset c+J_cQ_c.
\end{equation}
Taking into account \eqref{s-JQ} one gets immediately two nonnegative extensions for $S-c$:
\[
 S-c \subset Q_c^*Q_c \quad \text{and} \quad S-c \subset J_cJ_c^*.
\]
The adjoint $J_c^* \in \bL(\sH, \sK_c)$ of the companion relation $J_c$
is a closed relation by definition.
Let $\pi$ be the orthogonal projection from $\sK_c$ onto
\begin{equation}\label{jjqc}
\mul J_c^*=(\dom J_c)^\perp=(\ran Q_c)^\perp.
\end{equation}
Then it follows from \eqref{s-JQ} and \eqref{jjqc} that,
with the orthogonal operator part $(J_c^*)_{\rm reg}$,
one also has
\begin{equation}\label{foot}
 Q_c \subset (J_c^*)_{\rm reg}, \quad \mbox{where} \quad (J_c^*)_{\rm reg}= (I-\pi) \,J_c^*.
\end{equation}

\begin{definition}\label{ssDef}
Let $S \in \bL(\sH)$  be a semibounded relation
with lower bound $\gamma \in \dR$ and let $c \leq \gamma$.
Let $Q_c \in \bL(\sH, \sK_c)$ be a representing map for $\st(S)-c$
with companion relation $J_c \in \bL(\sK_c, \sH)$.
Then the sesquilinear form $\ss(S) \in \bF(\sH)$ is defined by
\begin{equation}\label{ss}
 \ss(S)[\varphi , \psi] = c(\varphi, \psi)+((J_c^*)_{\rm reg} \varphi, (J_c^*)_{\rm reg} \psi),
 \quad \varphi, \psi \in \dom \ss(S)=\dom J_c^*.
\end{equation}
Consequently, the form $\ss(S)$ is semibounded and closed.
\end{definition}

By definition $(J_c^*)_{\rm reg}$ is a closed operator
and it is clear from \eqref{foot} that
\[
 (Q_c)^{**}  \subset (J_c^*)_{\rm reg}.
\]
Hence, it follows from \eqref{forrep} and \eqref{ss} that
the closed form $\ss(S)$ extends the form $\st(S)$.
Therefore the form $\st(S)$ is closable; the representation
of the closure is immediate from \eqref{forrep};
cf. \cite{HS2023b}.
Observe, that the form $\ss(S)$ is bounded below by $c$.
In fact, $\ss(S)$ has lower bound $c$:
For $c<\gamma$ this follows from
\[
 \ker (J_c^*)_{\rm reg}=\ker J_c^* =(\ran J_c)^\perp=(\ran (S-c))^\perp=\ker (S^*-c)\neq \{0\},
\]
while for $c=\gamma$ this is holds by the inclusion $\st(S) \subset \ss(S)$.
These observations are summarized in the next proposition.

\begin{proposition}\label{cloclo}
The closed form $\ss(S) \in \bF(\sH)$ defined  in \eqref{ss}
extends the form $\st(S) \in \bF(\sH)$ in \eqref{s-ip}
and its lower bound is $c$.
Consequently, the form $\st(S)$ is closable and
\[
 \st(S) \subset \bar{\st}(S) \subset \ss(S).
\]
Furthermore, the closure $\bar{\st}(s) \in \bF(\sH)$ is given by
\begin{equation}\label{cst}
{\bar \st}(S)[\varphi, \psi]=c(\varphi, \psi)+(Q_c^{**} \varphi, Q_c^{**} \psi),
\quad \varphi, \psi \in \dom {\bar \st}(S)=\dom Q_c^{**}.
\end{equation}
\end{proposition}

As a corollary of Proposition \ref{cloclo} it is shown that the formula \eqref{s-ip}
in Definition \ref{defsss} holds under more general circumstances.

\begin{corollary}\label{lemmon}
Assume that  $\{\varphi, \varphi'\} \in S^{*}$
and that $\varphi \in \dom \bar{\st}(S)$. Then
\begin{equation}\label{newww}
\bar{\st}(S)[\varphi, \psi] =(\varphi', \psi)
\quad \mbox{for all}
 \quad \psi \in \dom \bar{\st}(S).
\end{equation}
\end{corollary}

\begin{proof}
Since $Q_c$ is closable with $\dom \bar \st(S)=\dom Q_c^{**}$ and
$Q_c^{**}\subset J_c^*$, one has
\begin{equation}\label{new-}
 (\varphi, \psi'-c \psi)=(Q_c^{**} \varphi, Q_c \psi),
\quad \varphi \in \dom Q_c^{**}, \,\, \{\psi, \psi'\} \in S.
\end{equation}
Now assume that in \eqref{new-} one also has  $\{\varphi, \varphi'\}  \in S^{*}$.
Then for all $\{\psi, \psi'\} \in S$ one sees by definition
$(\varphi', \psi)=(\varphi, \psi')$. Thus \eqref{new-}, with $\dom S=\dom Q_c$, reads as
\begin{equation}\label{newwww}
 (\varphi'-c \varphi, \psi)=(Q_c^{**} \varphi, Q_c \psi),
\quad \{\varphi, \varphi'\} \in S^*, \,\varphi \in \dom Q_c^{**}, \,\, \psi  \in \dom Q_c.
\end{equation}
Taking limits $\{\psi_n,Q_c \psi_n\}\to \{\psi,Q_c^{**}\psi\}$ in \eqref{newwww} one obtains
\[
(\varphi'-c \varphi, \psi)=(Q_c^{**} \varphi, Q_c^{**} \psi),
\quad \{\varphi, \varphi'\} \in S^*, \;\varphi,\, \psi \in \dom Q_c^{**},
\]
which leads to \eqref{newww}.
\end{proof}

The next result concerns the description of the (form) domain of $\bar{\st}(S)$
which is easily derived from the associated representing map $Q_c$.
This type of description goes back to Freudenthal \cite{F}.
There is also a useful description of $\ran Q_c^*$ by means of $S$
which determines the form $\st(S)$; cf.  \cite{H, HSS2010, Seb83}.
The contents of this theorem will be interpreted in terms
of Friedrichs extensions later in Section~\ref{Fried}.

\begin{theorem}\label{kerFr}
Let $S \in \bL(\sH)$ be semibounded with lower bound $\gamma \in \dR$,
let $c \leq \gamma$, and let $Q_c\in\bL(\sH,\sK')$ be a representing map for $\st(S)-c$.
Then $\varphi\in \dom {\bar \st}(S)$
if and only if there is a sequence $\{\varphi_n,\varphi_n'\}\in S$ such that
\[
  \varphi_n \to \varphi \quad \text{and} \quad (\varphi_n'-\varphi_m',\varphi_n-\varphi_m)\to 0.
\]
In addition, $\varphi\in \ker ({\bar \st}(S)-c)$
if and only if there exists
$\{\varphi_n,\varphi_n'\}\in S$ such that
\[
 \varphi_n \to \varphi \quad \text{and}\quad (\varphi_n'-c\varphi_n,\varphi_n)\to 0.
\]
Furthermore, $\varphi\in \ran Q_c^*$ if and only if there exists $C_\varphi <\infty$ such that
\begin{equation}\label{halfFr}
  |(\psi, \varphi)|^2 \leq  C_\varphi(\psi' , \psi) \quad \text{ for all }\; \{\psi,\psi'\}\in S-c.
\end{equation}
\end{theorem}

\begin{proof}
By Proposition~\ref{cloclo} one sees that $\dom {\bar \st}(S)=\dom Q_c^{**}$ and hence
the condition that $\varphi\in \dom{\bar \st}(S)$ is equivalent to
\[
 \varphi_n \to \varphi \quad \text{and} \quad \|Q_c(\varphi_n-\varphi_m)\|^2\to 0.
\]
It remains to apply \eqref{nuk} to get the description of $\dom {\bar \st}(S)$
in the first statement.

As to the second statement, again by Proposition~\ref{cloclo}
one has similarly that $\ker ({\bar \st}(S)-c)=\ker Q_c^{**}$.
Hence, $\varphi\in \ker Q_c^{**}$ if and only if the exists a sequence
$\{\varphi_n,\varphi_n'\}\in S$ such that $\varphi_n\to \varphi$ and $Q_c\varphi_n\to 0$.
Using \eqref{nuk} this second condition can be rewritten as
\[
 (\varphi_n'-c\varphi_n,\varphi_n)= \|Q_c \varphi_n\|^2 \to 0.
\]

Finally, to prove the last statement recall a general description of $\ran T^*$ of $T\in\bL(\sH,\sK)$,
see e.g. \cite{H, HSS2010, Seb83}: $\varphi\in\ran Q_c^*$ if and only if for some $c_\varphi<\infty$,
\begin{equation}\label{ranQ*}
  |(f,\varphi)| \leq c_\varphi \|Q_cf\| \;  \text{ for all }  f \in \dom Q_c.
\end{equation}
Now using \eqref{nuk} the estimate in \eqref{ranQ*} can be rewritten as
\begin{equation}\label{ranQ*2}
 |(f,\varphi)|^2 \leq c_\varphi^2\, \|Q_cf\|^2 = c_\varphi^2\, (f'-cf,f)
 \quad \text{ for all }\; \{f,f'\}\in S.
\end{equation}
Take $C_\varphi=c_\varphi^2$ and observe that $\{f,f'\}\in S$
is equivalent to $\{f,f'-cf\}\in S-c$.
It remains to replace $f$ by $\psi$ and $f'-cf$ by $\psi'$
to conclude that \eqref{ranQ*2} is equivalent to \eqref{halfFr}.
This completes the proof.
\end{proof}

The interest in the closure $\bar \st(S)$  of the form $\st(S)$ lies in the possibility
to uniquely associate a semibounded selfadjoint relation which represents $\bar \st(S)$
and extends the semibounded relation $S$.
To be more specific,  recall that for every closed semibounded
form $\st \in \bF(\sH)$ there exists a unique semibounded selfadjoint
relation $H \in \bL(\sH)$, such that
\[
\st[\varphi, \psi]=(\varphi', \psi), \quad \{\varphi, \varphi'\} \in H, \quad \psi \in \dom \st,
\]
see \cite{BHS,HS2023b}.
Moreover, any semibounded selfadjoint relation $H \in \bL(\sH)$ is obtained in this way;
cf. \cite[Lemma 6.3]{HS2023b}.
Hence, there is a one-to-one correspondence between the semibounded closed forms $\st \in \bF(\sH)$
and the semibounded selfadjoint relations  $H \in \bL(\sH)$ which is indicated by
\begin{equation}\label{conven}
\st=\st_H \quad \mbox{or} \quad H=H_\st,
\end{equation}
depending on the point of view.
Furthermore, it should be remembered that for any
semibounded selfadjoint relation $H \in \bL(\sH)$
one may define the form $\st(H) \in \bF(\sH)$ as in \eqref{s-ip} by
\[
\st(H)=(\varphi', \psi), \quad \{\varphi, \varphi'\}, \{\psi, \psi'\}\in H.
\]
Consequently, the form $\st(H)$ is semibounded and closable by Proposition \ref{cloclo},
so that its closure $\bar \st(H)$ is a closed semibounded form.
Since the forms $\bar{\st}(H)$ and $\st_H$ coincide on $\dom H$
it follows that for any semibounded selfadjoint relation $H \in \bL(\sH)$ one has
\begin{equation}\label{hihi}
 \st(H) \subset \bar \st(H)=\st_H,
\end{equation}
see  \eqref{conven} and \cite{BHS}.
As a consequence it is clear that for semibounded selfadjoint relations $H,K \in \bL(\sH)$
one has the equivalent statements
\begin{equation}\label{hihihi}
\st(H) \leq \st(K) \quad \Leftrightarrow  \quad \st_{H} \leq \st_{K} \quad \Leftrightarrow \quad H \leq K,
\end{equation}
see \cite{BHS}. These results will be tacitly used in the rest of this paper.
Note that for a semibounded relation $S \in \bL(\sH)$
the form $\st(S) \in \bF(\sH)$ in \eqref{s-ip} was denoted by $\st_S$ in \cite{BHS};
however, this last notation will not be used in this paper (unless $S$ is selfadjoint) as it may be confusing
with the above conventions \eqref{conven} and \eqref{hihi}.
At this stage it should be mentioned that for general semibounded forms one can also introduce
uniquely determined symmetric relations, see \cite[Theorem 6.1]{HS2023b} and, for special cases
in the present context, Section \ref{Fried} and
Section \ref{Kreintype}.

\medskip

Now return to the semibounded form $\st(S) \in \bF(\sH)$ from \eqref{s-ip}.
The closure of the form $\st(S)$ in \eqref{cst} induces a unique
semibounded selfadjoint relation denoted by $S_{\rm F}$
and the closed form $\ss(S)$ in Definition~\ref{ssDef} induces another unique
semibounded selfadjoint extension denoted by $S_{{\rm K},c}$.
It follows from  \eqref{s-JQQ} that $S_{\rm F}$ and $S_{{\rm K},c}$ are extensions of $S$.
They are usually called the Friedrichs extension and the Kre\u{\i}n type extension of $S$,
respectively. In the notation as discussed in \eqref{conven} one may write
\[
\st_{S_{\rm F}} =  {\bar \st}(S) \subset \ss(S)=\st_{S_{{\rm K},c}}
\quad \mbox{and} \quad \st_{S_{{\rm K},c}} \leq \st_{S_{\rm F}}.
\]
Since dealing with relations allows taking inverses,
 it is easy to show that $J_c^{-1}$ is a representing map for the nonnegative form $\st((S-c)^{-1})$
and that $Q_c^{-1}$ is its companion relation. This observation leads immediately to the connecting formula
\eqref{coddin}.  The following sections will be devoted to the details of these observations.

\section{Semibounded forms induced by semibounded relations}
 \label{symm}

First consider a general semibounded sesquilinear form with lower bound $\gamma$
and let $c \leq \gamma$ as discussed in \cite{HS2023b}.
If $Q_c \in \bL(\sH, \sK_c)$ and $Q_c' \in \bL(\sH,\sK_{c'})$
are representing maps for $\st-c$, then clearly there exists
a partial isometry $V \in \bB(\sK_c, \sK_c')$
with initial space $\cran Q_c$ and final space $\cran Q_c'$, such that
\begin{equation}\label{vnul+}
 Q_c'=V Q_c.
\end{equation}
Here the notation $\bB(\sH, \sK)$ stands for the class of all bounded
everywhere defined operators from the Hilbert space $\sH$ to the Hilbert space $\sK$.
Since $V \in \bB(\sK_c, \sK_c')$, it follows
that the corresponding adjoints are connected via
\[
(Q_c')^*=(Q_c)^*V^*;
\]
see e.g. \cite[Proposition~1.3.9]{BHS}.
Therefore the representation of the form $\st$
in terms of semibounded symmetric
or selfadjoint relations involving a representing map $Q_c$
is not affected by the nonuniqueness:
 \begin{equation}\label{kkt}
(Q_c')^*Q_c'=Q_c^*Q_c ,\quad
(Q_c')^*(Q_c')^{**}=Q_c^*(Q_c)^{**},
\end{equation}
for these results see \cite{HS2023b}.

\medskip

In the particular case that $\ran Q_c$ and $\ran Q_c'$ are dense subspaces,
the isometry $V$ in \eqref{vnul+} becomes a unitary operator in $\bB(\sK_c, \sK_c')$.
This \textit{minimality of the representing map} in \cite[Definition~2.3]{HS2023b}
can be obtained by replacing the spaces $\sK$ and $\sK'$ by
the closed subspaces $\cran Q_c$ and $\cran Q_c'$, respectively.
However, in the present work this assumption is not used and
the constructions are carried out without any minimality assumption.
This simplifies the treatment when a representing map is to be constructed
for a new form from forms whose representing maps are already given:
for instance, to define a representing map for a restriction of a given form,
see e.g. Section~\ref{eextr} (proof of Lemma~\ref{superheni}),
or when constructing representing maps involving the functional calculus of given forms;
cf. \cite[Sections~4,~5]{HS2023b} (sum decompositions of forms, mutual singularity of pairs of forms, etc.).

\medskip

Now let $S \in \bL(\sH)$ be semibounded with lower bound $\gamma \in \dR$
and let $c \leq \gamma$. The associated semibounded form $\st(S)$ in \eqref{s-ip}
leads to a representing map $Q_c$ via \eqref{nuk} and its companion relation $J_c$ in \eqref{s-J}
whose basic properties are now established.
The following identities are straightforward consequences from the definition:
\begin{equation}\label{s-J1}
\left\{
\begin{array}{ll}
\dom J_c=\ran Q_c, & \ker J_c=\{0\}, \\
\ran J_c=\ran (S-c), &
\mul J_c=\{ \varphi'-c \varphi :\, \{\varphi, \varphi'\} \in S, \,Q_c \varphi=0\}.
\end{array}
\right.
\end{equation}
In the next result the identity involving $\mul J_c$ will be expressed differently,
independent of the representing map $Q_c$. In particular, it follows that
$J_c$ is an operator when $S$ is densely defined, in which case $S$ is an operator.
The contents of the next theorem will be interpreted for
the Kre\u{\i}n type extensions later in Section~\ref{Kreintype}.

\begin{theorem}\label{lemin}
Let $S \in \bL(\sH)$ be semibounded with lower bound $\gamma \in \dR$
and let $c \leq \gamma$. Then the multivalued part of $J_c$ in \eqref{s-J} is given by
\begin{equation}\label{oold}
 \mul J_c= \ran (S-c) \cap \mul S^*.
\end{equation}
The element $\psi'\in \mul J_c^{**}$ if and only if there exists
$\{\varphi_n,\varphi_n'\}\in S$ such that
\begin{equation}\label{mulJ**}
 \varphi_n'-c\varphi_n \to \psi' \quad \text{and}\quad (\varphi_n'-c\varphi_n,\varphi_n)\to 0.
\end{equation}
Furthermore, $\psi'\in \dom J_c^*$ if and only if there exists $C_{\psi'} <\infty$ such that
\begin{equation}\label{domJ*}
  |(\varphi',\psi')|^2 \leq  C_{\psi'}(\varphi' , \varphi) \quad \text{ for all }\; \{\varphi,\varphi'\}\in S-c.
\end{equation}
\end{theorem}

\begin{proof}
The following equivalence is a consequence of \eqref{nuk}:
\[
 (\varphi' -c \varphi, \varphi)=0 \quad \Leftrightarrow \quad Q_c \varphi=0
\]
for all $\{\varphi, \varphi'\} \in S$. In other words,
\begin{equation}\label{src}
\mul J_c=\big\{ \varphi'-c\varphi \in \sH:\,
(\varphi'-c \varphi,\varphi)=0 \,\, \mbox{for some} \,\, \{\varphi,\varphi'\} \in S \big\}.
\end{equation}
Now \eqref{oold} can be easily proved; cf. \cite[Lemma~3.1]{HSSW2007}.

Assume that $\varphi'-c\varphi \in \ran (S-c) \cap \mul S^*$
for some $\{\varphi, \varphi'\} \in S$.
This implies that $(\varphi'-c \varphi, \varphi) = 0$
and therefore, by \eqref{src},  $\varphi'-c \varphi \in \mul J_c$.

Conversely, assume that $\varphi'-c\varphi \in \mul J_c$.
Therefore, by \eqref{src}, one sees that
$\{\varphi, \varphi'\} \in S$
and $(\varphi'-c \varphi, \varphi) = 0$.
Clearly with $\{\psi, \psi'\} \in S$ the Cauchy-Schwarz inequality
for the nonnegative form $\st(S)-c$ gives
\[
 |(\varphi'-c \varphi, \psi)|^2 \leq (\varphi'-c \varphi, \varphi)(\psi'-c\psi, \psi).
\]
Hence $(\varphi'-c \varphi, \psi)=0$ for all $\psi \in \dom S$, so that
$\varphi' -c \varphi \in (\dom S)^\perp =\mul S^*$.
Therefore, it follows that $\varphi' -c \varphi \in \ran (S-c) \cap \mul S^*$.

To prove \eqref{mulJ**} let $\psi'\in \mul J_c^{**}$. Using \eqref{s-J} this means that
there exists a sequence $\{\varphi_n,\varphi_n'\} \in S$ such that $Q_c\varphi_n\to 0$ and
$\varphi_n'-c\varphi_n \to \psi'$. Now apply \eqref{nuk} to see that
\[
 (\varphi_n'-c\varphi_n,\varphi_n)=\|Q_c \varphi_n\|^2\to 0.
\]
The converse statement is shown similarly.

To prove \eqref{domJ*} observe that by \eqref{s-J1} $J_c^{-1}$ is an operator and that
$\dom J_c^*=\ran (J_c^{-1})^*$.
Now proceeding as in the proof of Theorem~\ref{kerFr} one concludes that $\psi'\in \dom J_c^*$
if and only if there exists $c_{\psi'} <\infty$ such that
\[
  |(f'-cf,\psi')| \leq c_{\psi'} \|J_c^{-1}(f'-cf)\|= c_{\psi'} \|Q_cf\| \;  \text{ for all }  \{f,f'\}\in S.
\]
Again squaring this, using \eqref{nuk}, and then replacing $f$ by $\varphi$ and $f'-cf$ by $\varphi'$ yields
\[
 |(\varphi',\psi')|^2 \leq c_{\psi'}^2\, (\varphi,\varphi')
 \quad \text{ for all }\; \{\varphi,\varphi'\}\in S-c.
\]
Thus with $C_{\psi'}=c_{\psi'}^2$ one gets \eqref{domJ*}.
This completes the proof.
\end{proof}

Since the representing map $Q_c \in \bL(\sH, \sK_c)$ in \eqref{forrep}
is only determined up to a partial isometry,
its choice will influence the companion relation $J_c \in \bL(\sK_c, \sH)$ in \eqref{s-J}.
Thus it will have to be checked whether the choice of $Q_c$ influences
the relation $c +J_cQ_c$ in \eqref{s-JQQ} and the form $\ss(S)$ in \eqref{ss}.

\begin{lemma}\label{jjjj}
Let $S \in \bLR(\sH)$ be semibounded with lower bound $\gamma \in \dR$
and let $c \leq \gamma$. Assume that  $Q_c \in \bL(\sH, \sK_c)$
and $Q_c' \in \bL(\sH,\sK_{c'})$ are representing
maps for $\st(S)-c$ which are connected via $Q_c'=VQ_c$
by a partial isometry $V \in \bB(\sK_c, \sK_c')$ as in \eqref{vnul+}.
Let $J_c \in \bL(\sK_c, \sH)$ and $J_{c}' \in \bL(\sK_{c}', \sH)$
be the companion relations of $Q_c$ and
$Q_c'$, respectively. Then $J_c$ and $J_c'$ are connected by
 \begin{equation}\label{jqjq}
J_c' \,(V \uphar \ran Q_c) =J_c.
\end{equation}
In particular, one has
\begin{equation}\label{jqjqq}
S \subset c+J_c'Q_c'=c+J_c Q_c;
\end{equation}
here the inclusion is an equality if and only if $\ran (S-c) \cap \mul S^*=\mul S$.
Moreover, the adjoints are connected via
\begin{equation}\label{jqjq2}
 ((J_c')^{\ast})_{\rm reg} =V\, (J_c^{\ast})_{\rm reg}.
\end{equation}
\end{lemma}

\begin{proof}
Let $V$ be the indicated partial isometry.
Then the identity \eqref{jqjq} is obtained from \eqref{vnul+}
by means of the formula \eqref{s-J}.

The identity in \eqref{jqjqq} is a straightforward consequence of \eqref{vnul+} and \eqref{jqjq}.
The inclusion in \eqref{jqjqq} was already stated in \eqref{s-JQQ}.
Next observe that $\dom J_cQ_c=\dom S$.
Hence, the inclusion in \eqref{jqjqq} is an equality if and only if
$\mul S=\mul J_c Q_c=\mul J_c$.
Now the assertion follows from \eqref{oold} in Theorem~\ref{lemin}.

In order to prove the identity \eqref{jqjq2} note that
\begin{equation}\label{jqjq3}
(\mul (J_c')^*)^\perp =\cran Q_c'
=V \, \cran Q_c=V\, (\mul J_c^*)^\perp.
\end{equation}
Next assume that $\{f,f'\} \in V\,(J_c^*)_{\rm reg}$.
Then $\{f,h\} \in (J_c^*)_{\rm reg}$ and $f'= Vh$.
Since $J_c^*$ is closed $\{f,h\} \in (J_c^*)_{\rm reg}$ implies that
$\{f,h\} \in J_c^*$ and $h \in \cran Q_c$. Hence, $f' \in \cran Q_c'$,
and using \eqref{s-J} for all $\{\varphi, \varphi'\} \in S$ one gets
\[
\begin{split}
 0&=(h, Q_c \varphi)-(f, \varphi'-c \varphi)\\
 &=(V h, V Q_c \varphi)-(f, \varphi'-c \varphi) \\
&=(f', Q_c' \varphi)-(f, \varphi'-c \varphi),
\end{split}
\]
which gives $\{f, f'\} \in (J_c')^*$. Hence
the inclusion $V\, (J_c^*)_{\rm reg}  \subset ((J_c')^*)_{\rm reg}$
has been shown; cf. \eqref{jqjq3}.
For the reverse inclusion, let $\{f, f'\} \in ((J_c')^*)_{\rm reg}$.
In other words, $\{f,f'\}  \in (J_c')^*$ and $f'=Vh$
for some $h \in (\mul (J_c^*)^\perp$ by \eqref{jqjq3}.
Therefore, retracing the above steps shows $\{f,h\} \in (J_c^*)_{\rm reg}$,
and hence $\{f,f'\} \in V(J_c^*)_{\rm reg}$.
\end{proof}

\begin{corollary}\label{indep}
Under the circumstances of Lemma {\rm \ref{jjjj}} one has
\[
\left\{
\begin{array}{l}
 J_c' (J_c')^{*}=J_c' ((J_c')^{*})_{\rm reg}= J_c (J_c^{*})_{\rm reg}=J_c J_c^{*}, \\
  (J_c')^{\ast \ast} (J_c')^{\ast}=( ((J_c')^{\ast})_{\rm reg})^* ((J_c')^{\ast})_{\rm reg}
= ((J_c^{\ast})_{\rm reg})^*(J_c^{\ast})_{\rm reg}=J_c^{\ast \ast}J_c^{\ast} .
\end{array}
\right.
\]
\end{corollary}

\begin{proof}
 Recall that for any linear relation $T \in \bL(\sH, \sK)$ one has
\begin{equation}\label{teetee}
 TT^*=T(T^*)_{\rm reg} \quad \mbox{and}
 \quad T^{**}T^*=((T^*)_{\rm reg})^* (T^*)_{\rm reg},
\end{equation}
see \cite{HS2023seq} for similar results.
Therefore, due to the identies in \eqref{teetee} it suffices to verify that
\[
 J_c' ((J_c')^{*})_{\rm reg}= J_c (J_c^{*})_{\rm reg}
\quad \mbox{and} \quad
( ((J_c')^{\ast})_{\rm reg})^* ((J_c')^{\ast})_{\rm reg}
= ((J_c^{\ast})_{\rm reg})^*(J_c^{\ast})_{\rm reg}.
\]
However, these identities follow immediately from \eqref{jqjq} and \eqref{jqjq2}.
\end{proof}

\begin{remark}\label{sebsto}
Let $S \in \bL(\sH)$ be a semibounded relation in a Hilbert space $\sH$
with lower bound $\gamma \in \dR$  and $c \leq \gamma$.
Then $S-c \geq0$ and this implies that
\begin{equation}\label{SS0}
 (\varphi'-c\varphi, \psi'-c \psi)_{S-c}
 :=(\varphi'-c \varphi, \psi)=(\varphi, \psi' - c \psi),
 \quad
 \{\varphi,\varphi'\}, \, \{\psi,\psi'\} \in S,
\end{equation}
defines a semi-inner product on the linear subspace $\ran (S-c)$.
It is immediately clear that the linear subspace of null elements
of $\ran (S-c)$, provided with this semi-inner product, is given by
$\ran (S-c) \cap \mul S^*$; see \eqref{src} and Theorem \ref{lemin}.
The Hilbert space $\sH_{S-c}$ is defined as the completion of
\[
 \ran (S-c) / (\ran (S-c) \cap \mul S^*)
\]
with the inner product
\[
 ( [\varphi'-c \varphi],  [\psi'-c \psi] )_{S-c}
 := (\varphi'-c \varphi,  \psi'-c \psi )_{S-c},
 \quad \{\varphi,\varphi'\}, \, \{\psi,\psi'\} \in S,
\]
where $[\varphi'-c \varphi]$ and $[\psi'-c \psi]$
denote the relevant equivalence classes.
Then, clearly, the operator $q_c \in \bL(\sH, \sH_{S-c})$ given by
\[
 q_c \varphi = [ \varphi'-c \varphi ], \quad \{\varphi, \varphi'\} \in S,
\]
is well defined and, automatically, one has that $\ran q_c$
is dense in $\sH_{S-c}$.
It follows from the above identity \eqref{SS0}  that
\[
  (q_c \varphi,  q_c \psi)_{S-c}=(\st(S)-c)[\varphi, \psi],
  \quad \{\varphi,\varphi'\}, \, \{\psi,\psi'\} \in S.
\]
Thus the operator $q_c \in \bL(\sH, \sH_{S-c})$
acts as a minimal representing map for the form $\st(S)-c$.
This construction has been used for nonnegative relations
in \cite{HSSW2007}, see also \cite{HSSW2017, San} for the sectorial case.
In the case when $S$ is densely defined one has $\mul S^*=\{0\}$,
i.e., $S^*$ and, therefore, also $J_c$ is an operator.
When $S$ is a closed densely defined nonnegative operator, this construction
reduces to that of Sebesty\'en and Stochel appearing in \cite{SS}: in other words,
their construction produced a minimal representing map for the closable form $\st(S) \geq 0$.
\end{remark}

\section{A characterization of the Friedrichs extensions} \label{Fried}

Let $S \in \bLR(\sH)$ be a semibounded relation
with lower bound $\gamma \in \dR$.
Let $\st(S)$ be the corresponding closable form
with representing map $Q_c$ in \eqref{forrep}
and ${\bar \st}(S)$ be its closure with representing map $Q_c^{**}$ in \eqref{cst}.
Then the semibounded selfadjoint relation $S_{\rm F}
\in \bL(\sH)$ which is
uniquely determined by the semibounded closed form ${\bar \st}(S) \in \bF(\sH)$, is given by
\begin{equation}\label{fri}
S_{\rm F}=c+Q_c^* Q_c^{**},
\end{equation}
while the symmetric semibounded relation $S_{\rm f}
\in \bL(\sH)$ which is uniquely
determined by the closable form $\st(S) \in \bF(\sH)$, is given by
\begin{equation}\label{friw}
S_{\rm f}=c+Q_c^*Q_c,
\end{equation}
see \cite[Theorem 6.1 and Corollary 6.4]{HS2023b} for a characterization and details.
Moreover,  the relations $S_{\rm F}$ and $S_{\rm f}$
do not depend on the particular choice of the representing map $Q_c$;
see \eqref{kkt}.
It is a direct consequence of  \eqref{s-JQQ}
(using $J_c \subset Q_c^*$) and Lemma \ref{jjjj}
that $S_{\rm f}$ and  $S_{\rm F}$ are extensions of $S$, such that
 \begin{equation}\label{hardf+}
 S \subset S_{\rm f}=c+Q_c^*Q_c  \subset S_{\rm F}=c+Q_c^{*}Q_c^{**}.
\end{equation}
The multivalued parts $\mul S_{\rm f}=\mul S_{\rm F}=\mul Q_c^*=(\dom Q_c)^\perp$
satisfy
\begin{equation}\label{hardf++}
\mul S_{\rm f}=\mul S_{\rm F}=\mul S^*.
\end{equation}
Moreover, it follows from \cite[Corollary 6.4]{HS2023b}
that the lower bounds of $S_{\rm f}$ and $S_{\rm F}$
are given by
\[
 m(S_{\rm f})=m(S_{\rm F})=\gamma.
\]

The semibounded selfadjoint extension $S_{\rm F}$ in \eqref{fri}
is the well-known \textit{Friedrichs extension},
while the semibounded symmetric extension $S_{\rm f}$ in \eqref{friw}
will be called the \textit{weak Friedrichs extension}.
It is clear that for any linear extension $T$ of $S$ with $S \subset T \subset S_{\rm F}$
it follows that also  the inclusions $S \subset \overline{T} \subset S_{\rm F}$ hold,
where $\overline{T}$ stands for the closure of $T$.  Hence it follows that
$\st(S) \subset \st(\overline{T})\subset \st(S_{\rm F}) \subset \st_{S_{\rm F}}={\bar \st}(S)$,
which gives
\begin{equation}\label{foot1}
 {\bar \st}(\overline{T})={\bar \st}(S)
 \quad \mbox{and} \quad
 (\overline{T})_{\rm F}=S_{\rm F},
\end{equation}
so that the Friedrichs extension of such a
semibounded extension $T$ of $S$ remains the same.
In particular, one has $ (\overline{S})_{\rm F}=S_{\rm F}$.
From the definition the following translation invariance property
of $S_{\rm F}$ is immediate
\[
 (S-c)_{\rm F}=S_{\rm F}-c.
\]
The extensions $S_{\rm f}$ and $S_{\rm F}$ are intrinsically
characterized in the following theorem.

\begin{theorem}\label{friedext}
Let $S \in \bLR(\sH)$ be a semibounded  linear relation with lower bound
$\gamma \in \dR$ and let $\st(S)$ be the corresponding semibounded form in \eqref{s-ip}.
Then the following statements hold:
\begin{enumerate} \def\labelenumi{\rm(\alph{enumi})}
\item
The Friedrichs extension $S_{\rm F}
\in \bL(\sH)$ of $S$  in \eqref{fri}
is given by
\begin{equation}\label{kurtplus}
S_{\rm F}=\big\{ \{\varphi, \varphi'\} \in S^* :\, \varphi \in \dom \bar{\st}(S) \,\big\}.
\end{equation}
 If $H \in \bL(\sH)$ is a selfadjoint extension of $S$ with
\begin{equation}\label{kurt}
\dom H \subset \dom \bar \st(S),
\end{equation}
then $H=S_{\rm F}$.

\item
The weak Friedrichs extension
$S_{\rm f} \in \bL(\sH)$ of $S$  in \eqref{friw}
is  given by
\begin{equation}\label{kurtplus0}
S_{\rm f}=S \hplus (\{0\} \times \mul S^*).
\end{equation}
 If $H \in \bL(\sH)$ is a symmetric extension of $S$ with
\begin{equation}\label{kurt0}
\dom H \subset \dom S,
\end{equation}
then $H \subset S_{\rm f}$.
\item
The equality $S_{\rm f} = S_{\rm F}$ holds if and only if
 \begin{equation}\label{kurt1}
 \dom S=\cdom S \cap \dom S^*.
\end{equation}
In particular, \eqref{kurt1} holds if $\dom S$ is closed.
\end{enumerate}
\end{theorem}

\begin{proof}
(a) It is clear from $S \subset S_{\rm F}$ that $S_{\rm F} \subset S^*$.
Moreover,  it follows from
\eqref{fri} that $\dom S_{\rm F} \subset \dom Q_c^{**}=\dom \bar{\st}(S)$.
Thus $S_{\rm F}$ belongs to the right-hand side of \eqref{kurtplus}.
However, observe that
\[
 (\varphi', \varphi)=\bar{\st}(S)[\varphi, \varphi] \in \dR,
 \quad \{\varphi, \varphi'\} \in S^*,
 \quad \varphi \in \dom \bar{\st}(S),
\]
cf. Corollary \ref{lemmon}, so that the right-hand side of \eqref{kurtplus}
is a symmetric relation. Therefore, \eqref{kurtplus} has been shown.

Let $H\in \bL(\sH)$ be a selfadjoint extension of $S$ with  \eqref{kurt}. Then it follows from
\eqref{kurtplus} that $H \subset S_{\rm F}$. Since both relations are selfadjoint, one
concludes that $H=S_{\rm F}$.

(b)  It is clear from $S \subset S_{\rm f}$ that $S_{\rm f} \subset S^*$.
Moreover, it follows from  \eqref{friw}
that $\dom S_{\rm f} \subset \dom Q_c =\dom S$.
Hence one obtains
\begin{equation}\label{ja}
 S_{\rm f} \subset \big\{ \{\varphi, \varphi'\} \in S^* :\, \varphi \in \dom S \,\big\}.
\end{equation}
Now observe that
\[
\big\{ \{\varphi, \varphi'\} \in S^* :\, \varphi \in
\dom S \,\big\}=S \hplus (\{0\} \times \mul S^*) \subset S_{\rm f},
\]
so that \eqref{ja} gives \eqref{kurtplus0}.

Let $H \in \bL(\sH)$ be a symmetric extension of $S$ with  \eqref{kurt0}.
Then it follows from \eqref{ja} that $H \subset S_{\rm f}$.

(c) Thanks to \eqref{hardf+}, it is clear that $S_{\rm f} = S_{\rm F}$
if and only if $S_{\rm f}$ is selfadjoint,
which is equivalent to $\dom S=\cdom S \cap \dom S^*$; cf. \cite[Lemma 1.5.7]{BHS}.
\end{proof}

The linear subspace $\ran (S_{\rm F}-c)^{\half}$ can be characterized
in terms of the representing map $Q_c$ for $\st(S) \in \bF(\sH)$.
By \eqref{fri} one has $S_{\rm F}-c=Q_c^*Q_c^{**}$ and therefore
\[
\ran (S_{\rm F}-c)^{\half}=\ran Q_c^*.
\]
Thus the following result is clear from Theorem~\ref{kerFr}; cf. \cite{H,HSS2010, PS,SSi}.

\begin{corollary}\label{SFhalf}
Let $S\in \bLR(\sH)$ be as in Theorem {\rm \ref{friedext}} and let $S_{\rm F}$ be its Friedrichs extension.
Then $\varphi\in \ran (S_{\rm F}-c)^{\half}$ if and only if there exists $C_\varphi <\infty$ such that
\[
  |(\psi, \varphi)|^2 \leq  C_\varphi(\psi' , \psi) \text{ for all } \{\psi,\psi'\}\in S-c.
\]
\end{corollary}

Similarly, one has in terms of the representing map $Q_c$
\[
\dom (S_{\rm F}-c)^{\half}=\dom {\bar \st}(S)=\dom Q_c^{**},
\]
and this domain was  described already in Theorem \ref{kerFr};
cf. \cite[Remark~4.6]{HMS}, \cite[eq.~(7)]{H}, \cite[Corollary~5.3.4]{BHS}.
Additionaly, one has
\[
\ker (S_{\rm F}-c)=\ker ({\bar \st}(S)-c)=\ker Q_c^{**},
\]
which was also described in Theorem \ref{kerFr}.
In particular, if $S$ is nonnegative, i.e. $\gamma\geq 0$,
then one can take $c=0$;  cf. \cite[Proposition~4.10]{HMS}.

\medskip

The Friedrichs extension $S_{\rm F} \in \bL(\sH)$ guarantees the existence of at least
one semibounded selfadjoint extension of a semibounded relation $S \in \bL(\sH)$ and, in fact,
$S_{\rm F}$ has the same lower bound as $S$.
It is a consequence of  \eqref{hardf++} that $S_{\rm F}$ is an operator if and only if
$S$ is densely defined. In this case $S$ is automatically an operator and all
selfadjoint extensions of $S$ are also operators.
In the general nondensely defined case, $S_{\rm F}$
is the largest semibounded selfadjoint extension of $S$.
In particular, one should note that if $S \in \bL(\sH)$ is semibounded and selfadjoint
then $S_{\rm F}=S$.

\begin{corollary}\label{Frried}
Let $S \in \bLR(\sH)$ be a semibounded linear relation
with lower bound $\gamma \in \dR$.
Let $H \in \bL(\sH)$ be any  semibounded selfadjoint extension of $S$, then
\begin{equation}\label{INeq}
H \leq S_{\rm F}.
\end{equation}
 \end{corollary}

\begin{proof}
First recall that for any semibounded selfadjoint extension $H$ the corresponding
form $\st(H)$ is semibounded and that $\st(S) \subset \st(H)$. Since $H$ is selfadjoint
$\st(H)$ is closable and its closure is denoted by $\st_H$; see \eqref{conven}.
Thus it is clear that
\[
\st_{S_{\rm F}} ={\bar \st}(S) \subset {\bar \st}(H)
=\st_H \quad \mbox{and} \quad \st_H \leq \st_{S_{\rm F}},
\]
which gives the inequality \eqref{INeq}; see \cite{BHS}.
\end{proof}

\section{Forms and representing maps induced by inverses of semibounded relations}\label{dual}

In Section \ref{Fried} the explicit formula \eqref{fri} for the Friedrichs extension
of a semibounded relations $S$ and its extremality property in Corollary \ref{Frried}
was established. This was done  by means of the representing map $Q_c$ of the form $\st(S)-c$
using the fact that $\st_{S_{\rm F}}$ is the closure of the form $\st(S)$ and that
$\st_{S_{\rm F}}\subset \st_H$ for every semibounded selfadjoint extension $H$ of $S$.
 Among all semibounded selfadjoint extensions $H$ of $S$ the lower bound $m(H)$
can be any number below $\gamma$.
After fixing a real number $c\leq \gamma$ a natural question is whether there is
a minimal semibounded selfadjoint extension among all semibounded selfadjoint extensions
of $S$ satisfying $m(H) \geq c$. The most interesting case here is the case where $c=\gamma$.
This problem has a long history in the case of nonnegative operators $S$ and an affirmative answer
to the problem was given by M.G. Kre\u{\i}n in his famous papers \cite{Kr1,Kr2},
where all nonnegative selfadjoint extensions $H\geq 0$ in the densely defined case
were characterized via the operator interval
\[
 S_{\rm K} \leq H \leq S_{\rm F}.
\]
The minimal solution $S_{\rm K}$ is nowadays called the Kre\u{\i}n extension or the
Kre\u{i}n-von Neumann extension of $S$.
These results were extended to the case of nonnegative linear relations in \cite{CS}
using the key connection
\begin{equation}\label{CdeS}
  S_{\rm K}=((S^{-1})_{\rm F})^{-1}.
\end{equation}
The general semibounded situation has been treated in \cite{BHS} using the form $\st(S)$
and its closure to define $S_{\rm F}$ and the analog \eqref{coddin} of the formula \eqref{CdeS}
for the nonnegative relation $S-c$; see \cite[Sections~5.3 \& 5.4]{BHS}.
The new ingredient in the present approach is to use representing maps
for the nonnegative forms induced by $S-c$ and its inverse $(S-c)^{-1}$
with a fixed $c\leq \gamma$.
For the form $\st((S-c)^{-1})$ this leads to an explicit formula for the corresponding minimal extension
$S_{{\rm K},c}$ of $S$ in \eqref{coddin0} as will be seen in Section \ref{Kreintype}.

\medskip

Now let $S \in \bLR(\sH)$ be a semibounded linear relation with lower bound
$\gamma \in \dR$ and let $c \leq \gamma$ be fixed.
The relation $S-c$ is nonnegative and so is its inverse
\begin{equation}\label{bre00}
  (S-c)^{-1}
 =\big\{ \{\varphi'-c \varphi, \varphi\} :\, \{\varphi, \varphi'\} \in S \big\},
 \quad c \leq \gamma.
\end{equation}
The inverse $(S-c)^{-1}$ induces a nonnegative sesquilinear form
$\st((S-c)^{-1})$, whose domain is given by
\[
 \dom \st((S-c)^{-1})=\ran (S-c),
\]
and according to \eqref{s-ip}, \eqref{nuk}, and \eqref{bre00} one has for
all $\{\varphi, \varphi'\}, \{\psi,\psi'\}  \in S$,
\begin{equation}\label{bre0}
 \st((S-c)^{-1})[\varphi'-c \varphi,\psi'-c \psi]
 =(\varphi, \psi'-c \psi)=(\st(S)-c)[\varphi, \psi]
 =(Q_c \varphi, Q_c \psi).
\end{equation}

The next proposition identifies a representing map and
a corresponding companion relation for the nonnegative form $\st((S-c)^{-1})$ which,
in fact, can be expressed in terms of the corresponding data for the semibounded form $\st(S)$
introduced already in Section \ref{sect2}.

\begin{proposition}\label{invData}
Let $S \in \bL(\sH)$  be a semibounded relation
with lower bound $\gamma \in \dR$ and let $c \leq \gamma$.
Let $Q_c \in \bL(\sH, \sK_c)$ be a representing map for $\st(S)-c$
with the companion relation $J_c \in \bL(\sK_c, \sH)$ in \eqref{s-J}.
Then the inverse $J_c^{-1} \in \bL(\sH, \sK_c)$ of $J_c$ is an operator and
it is a representing map for the form $\st((S-c)^{-1})$,
\begin{equation}\label{quak1}
\st((S-c)^{-1})[\eta, \zeta]
=(J_c^{-1} \eta, J_c^{-1} \zeta), \quad \eta, \zeta \in \ran (S-c).
\end{equation}
The companion relation of the representing map $J_c^{-1}$ is given by $Q_c^{-1}\in \bL(\sK_c, \sH)$.
Moreover, with orthogonal operator part of the adjoint of $Q_c^{-1}$,
\begin{equation}\label{pipip}
 ((Q_c^{-1})^*)_{\rm reg}= (I-\pi) (Q_c^{-1})^*,
\end{equation}
where $\pi$ is the orthogonal projection from $\sK_c$ onto $\mul J_c^*$ as in \eqref{jjqc},
the sesquilinear form $\ss((S-c)^{-1}) \in \bF(\sH)$ {\rm (}cf. Definition~{\rm \ref{foot}{\rm )}}
is given by
\begin{equation}\label{sss}
 \ss((S-c)^{-1})[\eta,\zeta]
 =((Q_c^{-1})^*)_{\rm reg} \,\eta, ((Q_c^{-1})^*)_{\rm reg} \, \zeta),
 \quad \eta, \zeta \in \dom (Q_c^{-1})^*.
\end{equation}
It is well-defined, nonnegative, and closed.
\end{proposition}

\begin{proof}
According to \eqref{s-J} the inverse $(J_c)^{-1}$ has the form
\begin{equation}\label{s-J-}
(J_c)^{-1}=\big\{\, \{ \varphi'-c \varphi,
 Q_c \varphi\} :\, \{\varphi, \varphi'\} \in S  \,\big\}
\end{equation}
and as noted in \eqref{s-J1} $J_c^{-1}$ is an operator.
The formula \eqref{quak1} is now clear from \eqref{bre0}.

It follows from \eqref{s-J-}  that
\[
 Q_c^{-1}=\big\{\{J_c^{-1} (\psi'-c \psi), \psi\}:\, \{\psi, \psi'\} \in S \big\},
\]
so that $Q_c^{-1}$ is the companion relation (with $0$ as base point)
of the representing map $J_c^{-1}$.
The adjoint $(Q_c^{-1})^*$ of the companion relation $Q_c^{-1}$ is closed by definition
and due to the identity
\[
\mul (Q_c^{-1})^*=(\ran Q_c)^\perp =\mul  J_c^*,
\]
its orthogonal operator part is given by \eqref{pipip}.
Therefore, Definition \ref{ssDef} can be applied to the form $\ss((S-c)^{-1})$ and this yields
a well-defined nonnegative closed form defined on the linear space $\dom (Q_c^{-1})^*$
given by the formula \eqref{sss}.
\end{proof}

It follows from Proposition~\ref{invData} that the operator $J_c^{-1}$
and the relation $Q_c^{-1}$ form a dual pair
\begin{equation}\label{s-JQ+}
  J_c^{-1} \subset (Q_c^{-1})^*   \quad \mbox{or, equivalently,} \
  \quad Q_c^{-1} \subset (J_c^{-1})^*;
\end{equation}
cf. \eqref{s-JQ}.
Moreover, the relation $Q_c^{-1} J_c^{-1}$ is an extension of $(S-c)^{-1}$:
\begin{equation}\label{s-JQQ+}
 (S-c)^{-1} \subset Q_c^{-1} J_c^{-1},
\end{equation}
which is symmetric as follows from \eqref{s-JQ+}; cf. \eqref{s-JQQ}.
The analog of the pairing stated in \eqref{foot}, when applied
to the operator $J_c^{-1}$ and the relation $((Q_c)^{-1})^*$
in \eqref{s-JQ+}, reads as
\[
 J_c^{-1} \subset ((Q_c^{-1})^*)_{\rm reg}.
\]
Thus the form $\st((S-c)^{-1})$, as a restriction of the closed form $\ss((S-c)^{-1})$,
is closable and its closure is given by
\begin{equation}\label{foot11}
{\bar \st}((S-c)^{-1})[\eta, \zeta] = ((J_c^{-1})^{**} \eta, (J_c^{-1})^{**} \zeta),
\end{equation}
for all $\eta, \zeta \in \dom {\bar \st}((S-c)^{-1})=\dom (J_c^{-1})^{**}$.

\medskip

In the terminology of Section \ref{Fried}, applied to the closable form $\st((S-c)^{-1})$,
it follows from \eqref{quak1} that the weak Friedrichs extension
$((S-c)^{-1})_{\rm f}$ of $(S-c)^{-1}$ is given by
\begin{equation}\label{kkt1}
((S-c)^{-1})_{\rm f}  =(J_c^{-1})^{*} (J_c^{-1}),
\end{equation}
while, similarly, it follows from \eqref{foot11}
that the Friedrichs extension $((S-c)^{-1})_{\rm F}$
of $(S-c)^{-1}$  is given by
 \begin{equation}\label{kkt2}
((S-c)^{-1})_{\rm F} = (J_c^{-1})^{*} (J_c^{-1})^{**}.
\end{equation}
According to \eqref{s-JQQ+} these two nonnegative relations are extensions
of $(S-c)^{-1}$, they satisfy
\[
 (S-c)^{-1} \subset ((S-c)^{-1})_{\rm f} \subset ((S-c)^{-1})_{\rm F},
\]
and they all have the same lower bound.

\section{A characterization of the  Kre\u{\i}n type extensions}\label{Kreintype}

Let $S \in \bLR(\sH)$ be a semibounded relation with lower bound $\gamma \in \dR$
and let $\ss(S) \in \bF(\sH)$ be the corresponding closed form with representing map
$ (J_c^*)_{\rm reg}$ as in \eqref{ss}. Then the semibounded selfadjoint relation
$S_{{\rm K},c}$, determined by the closed form ${\ss}(S)$, is clearly given by
\begin{equation}\label{kren2}
S_{{\rm K},c}=c+((J_c^{*})_{\rm reg})^* (J_c^{*})_{\rm reg}=c+J_c^{**} J_c^{*},
\end{equation}
where the last identity follows from Corollary \ref{indep}.
By Corollary \ref{indep} the relation $S_{{\rm K},c}$
does not depend on the particular
choice of representing map $Q_c$; cf. \cite{HS2023b}.
It is a direct consequence of \eqref{s-JQQ}
(using $Q_c \subset J_c^*$ in \eqref{s-JQQ})  that $S_{{\rm K},c}$
is an extension of $S$, i.e.,
\[
S \subset S_{{\rm K},c}.
\]
The multivalued part $\mul S_{{\rm K},c}$ satisfies
\[
 \mul S_{{\rm K},c}=\mul J_c^{**},
\]
and the formula \eqref{mulJ**} in Theorem~\ref{lemin} offers a description for this.
By Proposition \ref{cloclo} the lower bound of $\ss(S)$ is $c$ and consequently $c$
is also the lower bound of $S_{{\rm K},c}$; cf. \cite{HS2023b}.
The semibounded selfadjoint extension $S_{{\rm K},c}$ is the so-called
\textit{Kre\u{\i}n type extension} relative to the point $c \leq \gamma$.
By Corollary \ref{Frried} one sees immediately that $S_{{\rm K},c} \leq S_{\rm F}$.
From the right-hand side of \eqref{kren2} it follows that
\[
S_{{\rm K},c}-c=J_c^{**} J_c^{*}
= \left( (J_c^{-1})^* (J_c^{-1})^{**} \right)^{-1}=\left( ((S-c)^{-1})_{\rm F} \right)^{-1},
\]
where for the last identity \eqref{kkt2} was used.
Therefore, one concludes that the Kre\u{\i}n type extension
$S_{{\rm K},c}$ can be expressed in terms of the
Friedrichs extension induced by the closable form $\st((S-c)^{-1})$:
\begin{equation}\label{codding}
 S_{{\rm K},c}= c+(((S-c)^{-1})_{\rm F})^{-1}.
\end{equation}
Let $T$ be any linear extension $T$ of $S$ with $S \subset T \subset S_{{\rm K},c}$.
In conjunction with \eqref{foot1}
it becomes clear that
\[
 {\bar \st}((\overline{T}-c)^{-1})={\bar \st}((S-c)^{-1})
 \quad \mbox{and} \quad
 (\overline{T})_{{\rm K},c}=S_{{\rm K},c},
\]
since $\clos (T-c)^{-1}=(\overline{T}-c)^{-1}$.
In particular, one has $ (\overline{S})_{{\rm K},c}=S_{{\rm K},c}$.
There is also a weak version of the Kre\u{\i}n type extension,
but it has to be approached in an indirect way (since the form $\ss(S)$ is closed).
Inspired by the definition in \eqref{kren2}
the semibounded symmetric relation $S_{{\rm k},c}$ will be introduced as follows
\begin{equation}\label{kren1}
S_{{\rm k},c}:=c+J_c (J_c^{*})_{\rm reg}=c+J_c J_c^{*},
\end{equation}
where the last identity follows from Corollary \ref{indep}.
By Corollary \ref{indep} the relation $S_{{\rm k},c}$ does not depend on the particular
choice of representing map $Q_c$.
It is a direct consequence of \eqref{s-JQQ} or, equivalently, of \eqref{s-JQQ+}
(using $Q_c \subset J_c^*$) that $S_{{\rm k},c}$
is an extension of $S$ and that, in fact,
\[
  S \subset S_{{\rm k},c}=c+J_c J_c^{*}\subset S_{{\rm K}, c}=c+J_c^{**} J_c^{*}.
\]
Moreover, \eqref{s-J1} together with Theorem \ref{lemin} implies that the multivalued parts satisfy
\[
\mul S_{{\rm k},c}=\ran (S-c) \cap \mul S^* \subset  \mul S_{{\rm K},c} = \mul J_c^{**}.
\]
Since $m(S_{{\rm K},c})=c$, it follows from \eqref{kren2} and \eqref{kren1}
that $m(S_{{\rm k},c})=c$.
The semibounded selfadjoint extension $S_{{\rm k},c}$ is called the
\textit{weak Kre\u{\i}n type extension} relative to the point $c \leq \gamma$.
From the right-hand side of \eqref{kren1} it follows that
\[
S_{{\rm k},c}-c=J_c J_c^{*}= \left( (J_c^{-1})^* J_c^{-1}\right)^{-1}
=\left( ((S-c)^{-1})_{\rm f} \right)^{-1},
\]
where for the last identity \eqref{kkt1} was used. Therefore one concludes
 \begin{equation}\label{frii1}
S_{{\rm k},c}=c+(((S-c)^{-1})_{\rm f})^{-1}.
\end{equation}
Although $S_{{\rm k},c}$ has been defined directly in \eqref{kren1},
it is now clear from
\[
 (S_{{\rm k},c}-c)^{-1}=((S-c)^{-1})_{\rm f},
\]
that $(S_{{\rm k},c}-c)^{-1}$ can actually be expressed
in terms of the weak Friedrichs extension induced
by the closable form $\st((S-c)^{-1})$.

\medskip

With \eqref{kren2} and \eqref{kren1} in mind, it is now possible
to translate Theorem \ref{friedext} to the present context.
Thus the extensions $S_{{\rm k},c}$ and $S_{{\rm K},c}$ are intrinsically
characterized in the following theorem. Remember for (a) and (b) that
\[
\ran (S-c)=\dom \st((S-c)^{-1}) \subset \dom {\bar \st} ((S-c)^{-1})
\]

\begin{theorem}\label{neumanext}
Let $S \in \bLR(\sH)$ be a semibounded linear relation
with lower bound $\gamma \in \dR$
and let $\st(S)$ be the corresponding semibounded form in \eqref{s-ip}.
Let $c \leq \gamma$.
 Then the following statements hold:
 \begin{enumerate} \def\labelenumi{\rm(\alph{enumi})}
 \item
The Kre\u{\i}n type extension $S_{{\rm K}, c} \in \bL(\sH)$
of $S$ in \eqref{kren2} is given by
\begin{equation}\label{kurttplus}
 S_{{\rm K},c}=\big\{ \{\varphi, \varphi'\} \in S^*:\,
 \varphi' -c \varphi \in \dom \bar \st((S-c)^{-1}) \big\}.
\end{equation}
 If $H \in \bL(\sH)$ is a selfadjoint extension of $S$ with
\[
 \ran (H-c) \subset \dom \bar \st((S-c)^{-1}),
\]
then $H=S_{{\rm K},c}$.

\item
The weak Kre\u{\i}n type extension
$S_{{\rm k},c} \in \bL(\sH)$
of $S$ in \eqref{kren1}  is  given by
\begin{equation}\label{kurttpluss}
 S_{{\rm k},c}=S \hplus \wh \sN_c(S^*).
 \end{equation}
 If $H \in \bL(\sH)$ is a symmetric extension of $S$ with
\[
 \ran (H-c) \subset \dom \st((S-c)^{-1})=\ran (S-c),
\]
then $H \subset S_{{\rm k},c}$.
\item
For all $c < \gamma$
\[
 S_{{\rm K},c}= \overline{S} \hplus \wh \sN_c(S^*),
\]
and the identity $S_{{\rm k},\gamma} = S_{{\rm K},\gamma}$ holds
if and only if
\begin{equation}\label{kurttplus+}
 \ran (S-\gamma)=\cran (S-\gamma) \cap \ran (S^*-\gamma).
\end{equation}
In particular the identity  $S_{{\rm k},\gamma} = S_{{\rm K},\gamma}$
holds if $\ran (S-\gamma)$ is closed.
\end{enumerate}
 \end{theorem}

\begin{proof}
By Theorem \ref{friedext} one sees that
\[
  ((S-c)^{-1})_{\rm F}
  =\big\{ \{\varphi, \varphi'\} \in ((S-c)^{-1})^* :\,
  \varphi \in \dom \bar{\st}(S-c)^{-1}) \,\big\}
\]
and, likewise,
\[
 ((S-c)^{-1})_{\rm f}=(S-c)^{-1} \hplus \left(\{0\} \times \ker (S^*-c) \right).
\]
It follows from \eqref{frii1} and \eqref{codding} that
\[
(S_{\rm k,c}-c)^{-1}=((S-c)^{-1})_{\rm f},
\quad \mbox{and} \quad
(S_{\rm K,c}-c)^{-1}=((S-c)^{-1})_{\rm F},
\]
and then, by Theorem \ref{friedext},
\eqref{kurttpluss} and \eqref{kurttplus} follow.
Also by Theorem  \ref{friedext} one has for $c \leq \gamma$ that
$S_{{\rm k},c} = S_{{\rm K},c}$ if and only if \eqref{kurttplus+} holds.
\[
 \ran (S-c)=\cran (S-c) \cap \ran (S^*-c)
\]
holds. For every $c \leq \gamma$ the inclusion
\[
 \overline{S} \hplus \wh \sN_c(S^*) \subset  S_{{\rm K},c}
\]
is clear, and for $c < \gamma$ there is actually equality
as both sides are selfadjoint; cf. \cite[Lemma 5.4.1]{BHS}.
\end{proof}

The linear subspace  $\dom (S_{{\rm K},c}-c)^{\half}$
can be characterized in terms of the representing map
$(J_c^*)_{\rm reg}$ for $\ss(S) \in \bF(\sH)$.
By \eqref{kren2} one has $S_{{\rm K},c}-c=J_c^{**}J_c^{*}$
and therefore
\[
\dom(S_{{\rm K},c}-c)^{\half}=\dom J_c^*.
\]
Thus the following result is clear  from Theorem~\ref{lemin},
cf. \cite{AN,AHSS,H,HMS, HSS2010}.

\begin{corollary}\label{SKhalf}
Let $S\in \bLR(\sH)$ be as in Theorem {\rm \ref{neumanext}}
and let $S_{{\rm K},c}$ be its Kre\u{\i}n type extension.
Then $\psi'\in \dom (S_{{\rm K},c}-c)^{\half}$ if and only if
there exists $C_{\psi'} <\infty$ such that
\begin{equation}\label{halfKr}
  |(\varphi',\psi')|^2 \leq  C_{\psi'}(\varphi' , \varphi)
  \quad \text{ for all }\; \{\varphi,\varphi'\}\in S-c.
\end{equation}
\end{corollary}

The description of $\dom (S_{{\rm K},c}-c)^{\half}$ can also be obtained from Corollary \ref{SFhalf}
by means of the identity \eqref{codding}.  Moreover, it follows from  \eqref{codding} that
\[
\ran (S_{{\rm K},c}-c)^{\half}=\dom (((S-c)^{-1})_{\rm F})^{\half}=\dom {\bar \st}((S-c)^{-1}),
\]
where the right-hand side can be described by means of Theorem \ref{kerFr}. 

\medskip

Let $S \in \bLR(\sH)$ be semibounded with lower bound $\gamma \in \dR$
and let $c \leq \gamma$.
The Kre\u{\i}n type extension $S_{{\rm K},c}$ of $S$ has lower bound $c$
and, in fact, it is the smallest semibounded selfadjoint
extension of $S$ which is bounded below by $c$.

\begin{corollary}\label{Krrein}
Let $S \in \bLR(\sH)$ be a semibounded linear relation with lower bound
$\gamma \in \dR$ and $c \leq \gamma$.
Let $H \in \bL(\sH)$ be a  semibounded selfadjoint extension of $S$,
then there is the equivalence
\[
   c \leq H  \quad \Leftrightarrow \quad
 \st_{S_{{{\rm K},c}}} \leq \st_{H}.
\]
\end{corollary}

\begin{proof}
($\Rightarrow$) Assume that $c \leq H$.  Since $S \subset H$,
it follows that  $(S-c)^{-1} \subset (H-c)^{-1}$.
Thus also $(H-c)^{-1} \leq ((S-c)^{-1})_{\rm F}$ (see \eqref{INeq}) and,
by antitonicity, one obtains
\[
(((S-c)^{-1})_{\rm F})^{-1} \leq ((H-c)^{-1})^{-1}=H-c,
\]
see \cite{BHS}, or by \eqref{codding} $S_{{\rm K},c} \leq H$.
Consequently, $\st_{S_{{\rm K},c}}\leq \st_{H}$.

($\Leftarrow$)  This implication is straightforward.
\end{proof}

For a semibounded relation $S \in \bL(\sH)$
it is clear that $S_{\rm F} \in \bL(\sH)$ is an operator if and only if $S$ is densely defined;
and in this case all selfadjoint extensions of $S$ are operators.
The question when $S_{{\rm K},c}\in \bL(\sH)$ is an operator will be addressed now;
note that then $S$ is automatically an operator.
First observe, that if $S$ is a closed (or closable) operator, then by Theorem \ref{neumanext}
\[
  \dom S_{{\rm K},c} = \dom \overline{S} + \ker(S^*-c),\quad c<\gamma,
\]
is dense in $\sH$, and the selfadjoint relation $S_{{\rm K},c}$ with $c<\gamma$
is an operator.
The next proposition deals with the case that $S$ is an operator
which is not necessarily closed or closable;
see also \cite[Proposition~4.9]{HMS}, \cite[Corollary~5.4.8]{BHS},
\cite[Theorem~4.1]{ST2021}, and \cite[Theorem~2.1]{ST2022}.
The limit criterion \eqref{ANcr} below (with $c=0$) was introduced in
\cite{AN} for a closed nonnegative operator $S$, which was then called
\textit{positively closable}.

\begin{proposition}
Let $S \in \bLR(\sH)$ be a semibounded operator with lower bound $\gamma \in \dR$
and let $c\leq \gamma$. Then the following conditions are equivalent:
\begin{enumerate}
\item[(i)] the Krein type extension $S_{{\rm K},c}$ is an operator or,
equivalently, the form $\ss(S)$ is densely defined;
\item[(ii)] there exists a Hilbert space $\sK$ and a closable operator $B\in\bLR(\sK,\sH)$
    such that $S=c+BB^*\upharpoonright \dom S$;
\item [(iii)] $S$ admits a selfadjoint operator extension $H$ with lower bound $m(H)=c$;
\item[(iv)] every sequence $\{\varphi_n\} \in \dom S$ satisfies the following property:
\begin{equation}\label{ANcr}
 ((S-c)\varphi_n, \varphi_n) \to 0\quad\text{and}\quad
  (S-c)\varphi_n \to \omega \quad\text{imply}\quad \omega=0.
\end{equation}
\end{enumerate}
\end{proposition}

\begin{proof}
According to the general representation theorem one has that
\[
\mul S_{{\rm K},c}=(\dom \ss(S))^\perp,
\]
so that $S_{{\rm K},c}$ is an operator
if and only if the form $\ss(S)$ is densely defined.

(i) $\Rightarrow$ (ii) Let $S_{{\rm K},c}$ be an operator.
By \eqref{kren2} it follows that $\mul J_c^{**}=\mul S_{{\rm K},c}$, so that
$J_c\in \bLR(\sK_c,\sH)$ is a closable operator.
Then by \eqref{oold} and Lemma~\ref{jjjj} one has $S=c+J_cQ_c$,
where $Q_c\in \bLR(\sH,\sK_c)$ is the representing map \eqref{forrep}
for the form $\st(S)-c$.
Now recall that $Q_c\subset (J_c^*)_{\rm reg}=(I-\pi)J_c^*$,
where $I-\pi$ is the orthogonal projection onto
$(\mul J_c^*)^\perp=\cran Q_c$ (see \eqref{jjqc})
and that $(J_c^*)_{\rm reg}$ is a closed operator which extends $Q_c$.
Therefore,
\[
  S=c+J_cQ_c=c+J_c((J_c^*)_{\rm reg}\upharpoonright \dom S)
  =c+J_cJ_c^*\upharpoonright \dom S,
\]
see Corollary \ref{indep}. Thus one can take $\sK=\sK_c$ and $B=J_c$.

(ii) $\Rightarrow$ (iii) Since $B$ is closable the representation for $B$ implies that
\[
 S\subset c+BB^*\subset c+B^{**}B^*,
\]
where by assumption $B^{**}B^*$ is a nonnegative selfadjoint operator.
Thus $S$ admits a selfadjoint operator extension $H=c+B^{**}B^*$ with lower bound $m(H)\geq c$.

(iii) $\Rightarrow$ (i) By the extremality property of Kre\u{\i}n type extensions
stated in Corollary \ref{Krrein} one has $\st_{S_{{{\rm K},c}}} \leq \st_{H}$ and,
equivalently, $S_{{{\rm K},c}} \leq H$, see \eqref{hihihi}.
Since $H$ is an operator the same is true for $S_{{{\rm K},c}}$.

(i) $\Leftrightarrow$ (iv) This is a direct consequence of the explicit formula \eqref{kren2}
and the description of $\mul J_c^{**}$ in Theorem~\ref{lemin}, since according to the condition
\eqref{ANcr} one has that $\mul J_c^{**}=\{0\}$.
\end{proof}

By the extremality property of Kre\u{\i}n type extensions
it is also clear that if $S_{{\rm K},\gamma}$ is an operator
then the same is true for every selfadjoint extension $S_{{\rm K},c}$ with $c<\gamma$.
Similarly, if $S_{{\rm K},\gamma}\in\bB(\sH)$
then also $S_{{\rm K},c}\in\bB(\sH)$ for all $c<\gamma$.

\section{Extremal semibounded extensions}\label{eextr}

Let $S \in \bLR(\sH)$ be a semibounded relation
with lower bound $\gamma \in \dR$
and let $c \leq \gamma$. Then $S$ has semibounded selfadjoint extensions
$S_{\rm F}$ and $S_{{\rm K},c}$,  which satisfy
 \begin{equation}\label{qqne}
\st(S) \subset {\bar \st}(S)=\st_{S_{\rm F}}  \subset \st_{S_{{\rm K},c}}.
\end{equation}
Moreover,  for any $c \leq \gamma$
a semibounded selfadjoint extension $H \in \bL(\sH)$ of $S$  satisfies
the following equivalence
\begin{equation}\label{frkr+}
   c \leq H  \quad \Leftrightarrow \quad
 \st_{S_{{\rm K},c}} \leq \st_{H} \leq \st_{S_{\rm F}},
\end{equation}
see Corollaries \ref{Frried},~\ref{Krrein}.
In fact, the form inequalities in \eqref{frkr+} imply for a
semibounded selfadjoint relation $H \in \bL(\sH)$
that it is automatically an extension of $S$; see \cite{BHS}.

\medskip

Theorem \ref{neumanext} shows that for $c<\gamma$
one has $m(S_{{\rm K},c})=c<m(S)=m(S_{\rm F})$.
In the case that $c=\gamma$ one still has
$S_{{\rm K},\gamma}\leq S_{\rm F}$. It is possible that
here equality holds and the next result contains some criteria
for this situation; cf. \cite[Theorem~4.14]{HMS},
\cite[Proposition~12]{H}, see also \cite[Corollary~5.3.10]{BHS}.

\begin{proposition}
Let $S \in \bL(\sH)$ be semibounded with lower bound $\gamma \in \dR$
and let $c \leq \gamma$. Then the following statements are equivalent:
\begin{enumerate} \def\labelenumi{\rm(\roman{enumi})}
\item the equality $S_{{\rm K},\gamma}=S_{\rm F}$ holds;
\item $\ker(S^*-c)\cap \dom (S_{{\rm K},\gamma}-\gamma)^{\half}=\{0\}$
for some {\rm (}equivalently for all{\,\rm )} $c<\gamma$;
\item for some {\rm (}equivalently for all{\,\rm )} $c<\gamma$ one has
\begin{equation}\label{limcrit}
  \sup_{\{\varphi,\varphi'\}\in S-\gamma} \dfrac{|(\varphi',h)|^2}{(\varphi',\varphi)}=\infty
  \quad \text{ for every }\; h\in \ker(S^*-\gamma).
\end{equation}
\end{enumerate}
\end{proposition}

\begin{proof}
For the equivalence of (i) and (ii) see \cite[Corollary~5.3.10]{BHS}.

To see the equivalence of (ii) and (iii) let $c<\gamma$ and let $h\in \ker(S^*-c)$.
Then \eqref{halfKr} shows that $h \in \dom (S_{{\rm K},\gamma}-\gamma)^{\half}$
if and only if there exists $C_h<\infty$ such that
\[
 |(\varphi', h)|^2 \leq  C_h(\varphi' , \varphi) \quad \text{ for all }\; \{\varphi,\varphi'\} \in S-c.
\]
Thus, $\ker(S^*-c)\cap \dom (S_{{\rm K},\gamma}-\gamma)^{\half}=\{0\}$
if and only if \eqref{limcrit} holds.
\end{proof}

In the rest of this section it is assumed that
$S_{{\rm K},\gamma}\neq S_{\rm F}$, in which case
there are also other semibounded selfadjoint extensions $H$ of $S$
between these two extreme extensions.
In particular, in this case $\st_{S_{\rm F}}  \subsetneqq \st_{S_{{\rm K},c}}$ and
the form inclusions in \eqref{qqne} and the inequalities in \eqref{frkr+}
suggest the class of extremal extensions
that will be introduced and discussed in this section;
see Definition \ref{extr} below.
The definition of this class will be motivated by the following lemma.

\begin{lemma}\label{superheni}
Let $S \in \bLR(\sH)$ be semibounded with lower bound $\gamma \in \dR$
and let $c \leq \gamma$. Then the following statements are equivalent:
\begin{enumerate} \def\labelenumi{\rm(\roman{enumi})}
\item
$H \in \bL(\sH)$ is a selfadjoint extension of $S$
which is bounded below by $c$ and which satisfies
\begin{equation}\label{frkr}
  \st_{S_{\rm F}} \subset \st_{H} \subset \st_{S_{{\rm K},c}}.
\end{equation}
\item
$H=R_c^{*}R_c^{**}+c$ for some $R_c \in \bL(\sH, \sK_c)$ which satisfies
\begin{equation}\label{qrj+}
Q_c \subset R_c   \subset (J_c^{*})_{\rm reg}.
\end{equation}
\end{enumerate}
\end{lemma}

\begin{proof}
(i) $\Rightarrow$ (ii)
Assume that $H$ is a semibounded selfadjoint extension of $S$,
which is bounded below by $c$,
and which satisfies $\st_{H} \subset \st_{S_{{\rm K},c}}$.
The inclusion $\st_{H} \subset \st_{S_{{\rm K},c}}$ means that
\[
 \st_{H}[\varphi, \psi]
 =c(\varphi, \psi)+((J_c^{*})_{\rm reg}\varphi, (J_c^{*})_{\rm reg}\psi),
\quad \varphi, \psi \in \dom \st_H \subset \dom J_c^*,
\]
see \eqref{kren2}.
Then the operator $R_c$ be defined by $R_c =(J_c^{*})_{\rm reg} \uphar \dom \st_{H}$
is a representing map for the closed form $\st_H$ which extend the representing map
$Q_c$ in \eqref{forrep}.
In particular, $R_c$ is a closed operator and $H=c+R_c^*R_c$.
Thus \eqref{qrj+} has been shown.

(ii) $\Rightarrow$ (i) It follows from \eqref{s-JQQ} and \eqref{qrj+} that
\[
S-c \subset J_cQ_c \subset R_c^{*}R_c  \subset R_c^{*}R_c^{**},
\]
and, hence,  $H=c+R_c^{*}R_c^{**}$
is a semibounded selfadjoint extension of $S$.
It is clear from the construction that $R_c^{**} \subset (J_c^*)_{\rm reg}$,
so that $\st(H) \subset \st(S_{{\rm K},c})$.
Since $H$ is a semibounded selfadjoint extension of $S$ it follows
automatically that $\st(S_{\rm F}) \subset \st(H)$.
Therefore, one sees \eqref{frkr}.
\end{proof}

The next lemma is helpful in the discussion of extremal extensions
in Definition \ref{extr}.
Recall from \eqref{hihi} that for a semibounded selfadjoint relation
$H \in \bL(\sH)$ one has $\st(H) \subset \bar\st(H)=\st_H$.

\begin{lemma}\label{coro}
Let $S \in \bLR(\sH)$ be semibounded with lower bound $\gamma \in \dR$.
Assume that $c \leq \gamma$ and let $H \in \bL(\sH)$
be a semibounded selfadjoint extension of $S$
which is bounded below by $c$. Then there are the following statements.
\begin{enumerate}\def\labelenumi{\rm(\alph{enumi})}
\item
For $\{h,h'\} \in S$, $\{f,f'\} \in H$, and  $f \in \dom J_c^*$ one has
 \begin{equation}\label{tennis0}
\begin{split}
(\st(H)-c)[f-h, f-h]&-\|(J_c^*)_{\rm op} f-Q_c h\|^2 \\
& \hspace{1cm}=(\st(H)-c)[f,f]- \| (J_c^*)_{\rm reg} f\|^2.
\end{split}
\end{equation}
\item
For $\{f,f'\} \in H$ and $\{h,h'\} \in S$ one has
\begin{equation}\label{tennis1}
(\st(H)-c)[f-h, f-h] \geq  \|(J_c^*)_{\rm reg} f-Q_c h\|^2.
\end{equation}
\item
If $\st_H\subset \st_{S_{{\rm K},c}}$,
then for $\{f,f'\} \in H$ and $\{h,h'\} \in S$:
\begin{equation}\label{tennis2}
(\st(H)-c)[f-h, f-h]= \|(J_c^*)_{\rm reg} f-Q_c h\|^2.
\end{equation}
\end{enumerate}
 \end{lemma}

\begin{proof}
(a) Since $\{h,h'\} \in S$ there is the identity
\begin{equation}\label{2c}
(h'-ch, h)=(Q_c h, Q_c h).
\end{equation}
The assumption $f \in \dom J_c^*$
implies $\{f,(J_c^*)_{\rm reg} f\} \in J_c^*$.
Recall that for $\{h,h'\} \in S$ one has $\{Q_c h, h'-ch\} \in J_c$.
Hence, one obtains
\begin{equation}\label{3c}
 (h'-ch , f)=(Q_c h, (J_c^*)_{\rm reg} f).
\end{equation}
Since $\{f,f'\} \in H \subset S^*$ and $\{h,h'\} \in S$,
one has $(f',h)=(f,h')$, so that by \eqref{3c},
\begin{equation}\label{4c}
 (f'-cf, h)=(f, h'-ch)=((J_c^*)_{\rm reg} f, Q_c h).
\end{equation}
Thus, it follows from
\eqref{2c}, \eqref{3c}, and \eqref{4c} that
\[
\begin{split}
&(\st(H)-c)[f-h, f-h]  \\
&\hspace{1cm}
=(f'-cf -(h'-ch), f-h) \\
&\hspace{1cm}
=(f'-cf, f)-(f'-cf, h)-(h'-ch, f)+(h'-ch, h) \\
&\hspace{1cm} =(f'-cf, f) - ((J_c^*)_{\rm reg} f, Q_c h)
                               -(Q_c h, (J_c^*)_{\rm reg} f) + (Q_c h, Q_c h)\\
&\hspace{1cm}=(f'-cf, f)-((J_c^*)_{\rm reg} f, (J_c^*)_{\rm reg} f)
                           + \|(J_c^*)_{\rm reg} f-Q_c h\|^2, \\
\end{split}
\]
which gives the identity \eqref{tennis0}, since $(\st(H)-c)[f,f]=(f'-cf,f)$.

(b) and (c) Since $c\leq m(H)$, Corollary \ref{Krrein} shows
that $\st_{S_{{\rm K},c}} \leq \st_H$.
Then one sees that $\{f,f'\} \in H$ implies that $f \in \dom S_{{\rm K},c}=\dom J_c^*$ and
\begin{equation*}
  ((J_c^*)_{\rm reg}f, (J_c^*)_{\rm reg}f)=(\st(S_{{\rm K},c})-c)[f,f] \leq (\st(H)-c)[f,f],
\end{equation*}
with equality if $\st_{H} \subset \st_{S_{{\rm K},c}}$.
 Thus  \eqref{tennis1} and \eqref{tennis2} follow from \eqref{tennis0}.
 \end{proof}

The identity \eqref{tennis2}  in Lemma  \ref{coro} is the basis for the following definition.

\begin{definition}\label{extr}
Let $S \in \bLR(\sH)$ be semibounded with lower bound $\gamma \in \dR$
and let $c \leq \gamma$.
A selfadjoint extension $H \in \bL(\sH)$ of $S$ is called
\textit{extremal} with respect to $c$ if for all $\{f,f'\} \in H$ one has
\begin{equation}\label{extremal}
\inf \big\{ (f'-h'-c(f-h), f-h ) :\, \{h, h' \} \in S \big\}=0.
\end{equation}
\end{definition}

Let $H$ be a selfadjoint extension of $S$
which is extremal with respect to $c$.
Then it is clear from the definition that
for all $\{f,f'\} \in H$ and $\{h,h'\} \in S$ one has
\[
 (f'-h'-c(f-h), f-h) \geq 0.
\]
By taking $\{h,h'\}=\{0,0\}$, it follows
that $H$ is automatically bounded below by $c$.
Thus one may write \eqref{extremal} in the equivalent form
\begin{equation*}
\inf \big\{\, (\st(H)-c)[f-h, f-h] :\, \{h, h' \} \in S \big\}=0.
\end{equation*}

\begin{theorem}\label{extrr}
Let $S \in \bL(\sH)$ be a semibounded relation
with lower bound $\gamma \in \dR$ and let $c \leq \gamma$.
Then the following statements are equivalent:
\begin{enumerate} \def\labelenumi{\rm(\roman{enumi})}
\item
$H \in \bL(\sH)$ is a semibounded selfadjoint extension
of $S$ which is bounded below by  $c$
and which satisfies  \eqref{frkr};
\item
$H \in \bL(\sH)$ is a  semibounded selfadjoint extension
of $S$ which is extremal with respect to $c$.
\end{enumerate}
\end{theorem}

\begin{proof}
(i) $\Rightarrow$ (ii)
This follows from \eqref{tennis2} in Lemma \ref{coro}
by observing that $\ran Q_c$ is dense in $(\mul J_c^*)^\perp$.

(ii) $\Rightarrow$ (i)
Let $H \in \bL(\sH)$ be a selfadjoint extension of $S$
which is extremal with respect to $c$.
Then $H$ is bounded below by $c$ and by Corollary \ref{Krrein}
it follows that $\st_{S_{{\rm K},c}} \leq \st_{H}$.
Therefore one may apply \eqref{tennis1} in Lemma \ref{coro}
with $\{f,f'\} \in H$ and $\{h,h'\} \in S$:
\begin{equation*}
(f'-h'-c(f-h), f-h ) \geq  \|(J_c^*)_{\rm reg} f-Q_c h\|^2.
\end{equation*}
Therefore, both sides of the identity \eqref{tennis0} are nonnegative
for all $\{f,f'\} \in H$ and $\{h,h'\} \in S$.
Note that the element $\{h,h'\} \in S$
appears only in the left-hand side of \eqref{tennis0}.

Now let $\{f,f'\} \in H$ and choose $\epsilon >0$.
Since $H$ is extremal, it follows from \eqref{extremal} that
there exists $\{h,h'\} \in S$ such that.
\begin{equation}\label{hh0}
   (f'-h'-c(f-h), f-h) < \epsilon.
\end{equation}
From  \eqref{hh0} one concludes for the left-hand side of \eqref{tennis0}
\begin{equation*}
\begin{split}
& 0 \leq (\st(H)-c)[f-h, f-h] - \|(J_c^*)_{\rm reg} f-Q_c h\|^2 \\
& \hspace{4.2cm} \leq (\st(H)-c)[f-h, f-h] < \epsilon.
\end{split}
\end{equation*}
Therefore, it follows from the right-hand side of \eqref{tennis0}
 that for any $\{f,f'\} \in H$ and any $\epsilon >0$:
 \[
 0 \leq (\st(H)-c)[f,f]- (\st(S_{{\rm K},c})-c)[f, f] < \epsilon.
\]
Hence $ \st(H)[f,f]=\st(S_{{\rm K},c})[f, f]$, $\{f,f'\} \in H$,
and by polarization one obtains
\[
 \st(H)[f,g]=\st(S_{{\rm K},c})[f, g], \quad \{f,f'\}, \{g,g'\} \in H.
\]
Since $\dom  \st(H) \subset \dom S_{{\rm K},c}$, this implies
that $\st(H) \subset \st(S_{{\rm K},c})$ and,
since $H$ is a semibounded selfadjoint extension of $S$ one also has
$\st(S_{\rm F}) \subset \st(H)$. Thus \eqref{frkr} is satisfied.
\end{proof}

Definition \ref{extr} goes back to Arlinski\u{\i} and Tsekanovski\u{\i} \cite{ArTs88}, see also \cite{A88,Ar3}.
The results in this section are based on \cite{AHSS,HSSW2007}, where also
an earlier form of Theorem \ref{extrr} can be found.

\section{Symmetric relations whose domain and range are orthogonal}\label{ortho}

Let $S \in \bL(\sH)$ be a linear relation.
Then the \textit{numerical range} $\cW(S)$ of $S$ is defined by
\[
\cW(S) = \big\{(\varphi', \varphi):\, \{\varphi,\varphi'\}\in S:\, \|\varphi\| = 1\,\big\} \subset \dC,
\]
when $\dom S \neq \{0\}$, and by $\{0\}\subset\dC$
if $\dom S = \{0\}$, i.e., if $S$ is purely multivalued.
For linear relations $S \in \bL(\sH)$ the numerical range
$\cW(S)$ is a convex set; see \cite[Proposition~2.18]{HdSSz09}.
Moreover, it is clear that all eigenvalues in $\dC$ of $S$ belong
to $\cW(S)$.
Clearly, the numerical range of the inverse of $S$ is given by
\[
\cW(S^{-1}) = \{ \lambda\in \dC: \overline\lambda \in \cW(S)\,\}.
\]
The interest in this section is in linear relations $S \in \bL(\sH)$
for which $\cW(S)=\{0\}$; this class was discussed earlier in \cite{C}.
Here is a characterization of it; see \cite{HLS} and also \cite{RC}.

\begin{lemma}\label{numranzero}
For every linear relation $S \in \bL(\sH)$ the following statements are equivalent:
\begin{enumerate} \def\labelenumi{\rm(\roman{enumi})}
\item $\cW(S)=\{0\}$;
\item $\dom S\perp \ran S$.
\end{enumerate}
\end{lemma}

Now assume that $\cW(S)=\{0\}$. Then the linear relation $S$ is nonnegative
and the form $\st(S) \in \bF(\sH)$, determined by $S$ in \eqref{s-ip},
is trivial by Lemma \ref{numranzero}:
\[
 \st(S)[\varphi,\psi]=(\varphi',\psi)=0, \quad \{\varphi, \varphi'\},\, \{\psi, \psi'\}\in S.
\]
Thus $\st(S) \in \bF(\sH)$ is the null form defined on $\dom S$:
\[
 \st(S)=0, \quad \dom \st(S)=\dom S.
\]
This shows that $\gamma=0$ and, therefore, in \eqref{forrep}
one may choose $c=0$ and a representing map $Q \in \bL(\sH)$ by
\begin{equation}\label{q9}
Q=\dom S  \times \{0\}.
\end{equation}
Clearly,  the corresponding companion relation $J \in \bL(\sH)$, given by \eqref{s-J}, reduces to
\[
J=\big\{\, \{ 0, \varphi' \} :\, \{\varphi, \varphi'\} \in S  \,\big\} =\{0\} \times \ran  S,
\]
so that
\[
 J^*=(\ran S)^\perp \times \sH=\ker S^* \times \sH.
\]
Since $\mul J^*=\sH$, it follows that
\[
(J^*)_{\rm reg}=\ker S^* \times \{0\}.
\]
Hence the form $\ss(S)$ defined in \eqref{ss} is given by
\[
 \ss(S)[\varphi, \psi]=0,  \quad \varphi, \psi \in \dom \ss(S)=\ker S^*.
\]
Note that  \eqref{s-JQ} implies $\dom S \subset \ker S^*$ and $\ran S \subset \mul S^*$,
which can be seen immediately from the assumption $\cW(S)=\{0\}$.
It is clear from these observations that $\st(S) \subset \ss(S)$.

\begin{lemma}\label{nie}
Let $S \in \bL(\sH)$ be a linear relation such that $\cW(S)=\{0\}$.
Then the Friedrichs extension of $S$ is given by
\begin{equation}\label{frie}
S_{\rm F}=\cdom S \times \mul S^*=\cker S \times \mul S^*,
\end{equation}
and the Kre\u{\i}n type extension $S_{\rm K}=S_{{\rm K},0}$ is given by
\begin{equation}\label{krei}
S_{\rm K}=\ker S^* \times \cran S=\ker S^* \times \cmul S.
\end{equation}
\end{lemma}

\begin{proof}
It follows from the representation \eqref{q9} that
\[
 Q^{**}=\cdom S \times \{0\} \quad \mbox{and} \quad Q^*= \sH \times \mul S^*.
\]
By \eqref{fri} one has $S_{\rm F}=Q^*Q^{**}$, which gives \eqref{frie}.
 Likewise, one has
\[
((J_c^*)_{\rm reg})^*=\sH \times \cran S.
 \]
By \eqref{kren2} one has $S_{\rm K}= ((J_c^*)_{\rm reg})^* (J_c^*)_{\rm reg}$,
which gives \eqref{krei}.
\end{proof}

Thus one sees that both nonnegative selfadjoint extensions $S_{\rm F}$ and $S_{\rm K}$
are singular relations.
It follows from Lemma \ref{nie} that $S_{\rm F}=S_{\rm K}$ if and only if
\[
 \ker S^*=\cdom S \quad \mbox{and} \quad \mul S^*=\cran S.
\]
Furthermore,
\[
 S_{\rm F} \cap S_{\rm K}= \cker S \times \cran S,
\]
which implies that $S_{\rm F} \cap S_{\rm K}=\overline{S}$.
In addition, observe that the description in Corollary \ref{SFhalf}, when applied
to $S$ with $\cW(S)=\{0\}$, shows that
\[
 \ran S_{\rm F}^{\half} =(\dom S)^{\perp}=\mul S^*=\mul S_{\rm F},
\]
which agrees with \eqref{frie}. Similarly Corollary~\ref{SKhalf} yields
\[
 \dom S_{\rm K}^{\half}=(\ran S)^\perp=\ker S^*=\ker S_{\rm K},
\]
and this agrees with \eqref{krei}.

\medskip

Notice that if $\cW(S)=\{0\}$, then by Lemma \ref{nie} it follows that
$\cW(S_{\rm F})=\{0\}$ and, likewise, $\cW(S_{\rm K})=\{0\}$.
This observation gives a connection to the nonnegative extremal extensions of $S$
(i.e., extremal with respect to $0$).

\begin{proposition}
Let $S \in \bL(\sH)$ be nonnegative relation such that $\cW(S) = \{0\}$.
Let $H \in \bL(\sH)$ be a selfadjoint extension of $S$.
Then the following conditions are equivalent:
\begin{enumerate}\def\labelenumi{\rm(\roman{enumi})}
\item $\cW(H) = \{0\}$;
\item  $H$ is a nonnegative extremal extension of $S$.
\end{enumerate}
\end{proposition}

\begin{proof} The assumption $\cW(S)=\{0\}$ implies that $\cW(S_{\rm K})=\{0\}$.
In other words, $\dom S_{\rm K} \perp \ran S_{\rm K}$. Hence the form $\st(S_{\rm K})$
corresponding to $S_{\rm K}$ is the zero form on $\dom S_{\rm K}$.

(i) $\Rightarrow$ (ii) Assume that $\cW(H) = \{0\}$. Then clearly $\st(H)$ is the zero form
and $\st_H=0$. Since the assumption $\cW(H)=\{0\}$ implies that $\cW(S)=\{0\}$,
it follows from Lemma \ref{nie} that $\cW(S_{\rm F})=0$ and $\cW(S_{\rm K})=0$.
Thanks to $H$ being a nonnegative selfadjoint extension of $S$, Corollary \ref{Krrein}
implies that  $\st_{S_{{{\rm K}}}} \leq \st_{H}$,
which gives  $\dom \st_H \subset \dom \st_{S_{{{\rm K}}}} $.
Recall that $\st_{S_{\rm F}}  \subset \st_H $.
All these forms are null forms and one obtains
$\st_{S_{\rm F}}  \subset \st_H \subset \st_{S_{{{\rm K}}}}$.
Now Theorem \ref{extr} shows that $H$ is extremal.

(ii) $\Rightarrow$ (i) Let $H$ be an extremal extension of $S$.
Then by Theorem \ref{extr} one has $\st(H) \subset \st(S_{\rm K})$.
Hence $\st(H)$ is the zero form on $\dom \st(H)$. In particular, it follows that $\cW(H) = {0}$.
\end{proof}

\section{Symmetric relations defined by semibounded quadratic forms}\label{converse}

In this section attention is paid to the representation of a general semibounded form
$\st \in \bF(\sH)$ as given in \cite{HS2023b}. It is known that such a form induces
a semibounded relation $S_\st \in \bL(\sH)$ and a semibounded
selfadjoint relation $\wt A_\st \in \bL(\sH)$ which extends $S_\st$.
The relation $\wt A_\st$  is the Friedrichs extension of
an, in general,  nontrivial  extension of $S_\st$.
The semibounded  relation $S_\st$ itself gives rise to a semibounded form
$\st(S_\st)$, which induces a Friedrichs extension of $S_\st$.

\medskip

The general semibounded form $\st \in \bF(\sH)$ is given by
\begin{equation}\label{fortt0}
 \st[\varphi, \psi]=c(\varphi, \psi)+ (Q\varphi, Q \psi), \quad \varphi, \psi \in \dom Q,
\end{equation}
for some $c \in \dR$ and a representing map $Q \in \bL(\sH,\sK)$.
Note that the form $\st$ is closable (closed) or singular if and only if the
operator $Q$ is closable (closed) or singular, respectively.
As mentioned above, the form $\st$ induces
a semibounded selfadjoint relation $\wt A_\st \in \bL(\sH)$
and a semibounded relation $S_\st \in \bL(\sH)$, and they are given by
\begin{equation}\label{fortt1}
 \wt A_\st=c+Q^* Q^{**} \quad \mbox{and} \quad S_\st=c+Q^*Q,
\end{equation}  
which are characterized in \cite[Theorem 6.1]{HS2023b}, see also \cite[Theorem~3.2]{AtE1}. 
It follows from \eqref{fortt1} that
\[
 S_\st \subset \wt A_\st.
\]
Denote the regular
part of $Q$ by $Q_{\rm reg}$ and the regular.
part of $\st$ by $\sr$, so that
\begin{equation}\label{fortt00}
 \sr[\varphi, \psi]=c(\varphi, \psi)+ (Q_{\rm reg}\varphi, Q_{\rm reg} \psi),
 \quad \varphi, \psi \in \dom Q,
\end{equation}
is a regular (closable) form in $\bF(\sH)$; its closure is given by
\[
 \bar \sr[\varphi, \psi]=c(\varphi, \psi)+ (Q_{\rm reg}^{**}\varphi, Q_{\rm reg}^{**} \psi),
 \quad \varphi, \psi \in \dom Q^{**}.
\]
Then, likewise, there are
a semibounded selfadjoint relation $\wt A_\sr \in \bL(\sH)$
and a semibounded relation $S_\sr \in \bL(\sH)$, given by
\begin{equation}\label{mmainn}
S_{\sr}=c+(Q_{c,\rm reg})^*Q_{c, \rm reg}
\quad \mbox{and} \quad
\wt A_{\sr}=c+(Q_{c, \rm reg})^*(Q_{c, \rm reg})^{**},
\end{equation}
and they satisfy
\[
S_{\st} \subset S_{\sr} \quad \mbox{and} \quad
\wt A_{\st}=\wt A_{\sr}=A_{\bar \sr};
\]
where $A_{\bar \sr} \in \bL(\sH)$ is the semibounded selfadjoint relation
corresponding to the
closed form $\bar \sr \in \bF(\sH)$; see \cite[Theorem~6.5]{HS2023b}.
It is clear from \eqref{mmainn} that $\wt A_{\st}=\wt A_{\sr}=A_{\bar \sr}$ 
is the Friedrichs extension
of the semibounded relation $S_{\sr} \in \bL(\sH)$; see Section \ref{Fried}.

\medskip

Next consider the connection of the above facts with the constructions
in the present paper.
Let the semibounded form $\st \in \bF(\sH)$
be given by \eqref{fortt0} and let the
corresponding semibounded relation $S_\st$ be as in \eqref{fortt1}.
Then this relation $S_\st$ induces a semibounded form
$\st(S_\st) \in \bF(\sH)$ as in Definition \ref{defsss}.
First observe that the domain of $S_\st$ in \eqref{fortt1}
satisfies $\dom S_\st=\cR$,
when the linear subspace $\cR \subset \dom Q$ is defined by
\begin{equation*}
 \cR=\big\{ \varphi \in \dom Q:\, Q \varphi \in \dom Q^* \big\},
\end{equation*}
so that in general one has
\[
 \{0\} \subset \cR \subset \dom Q.
\]
Define the operator $\dot{Q} \in \bL(\sH, \sK)$
as the restriction of $Q$ to the linear subspace $\cR \subset \dom Q$:
\begin{equation*}
 \dot{Q}=Q \uphar \cR, \quad \mbox{where} \quad \dom \dot{Q}=\cR.
\end{equation*}
Then it follows directly from the definition that with $\varphi \in \cR$ one has
\[
Q\varphi \in \dom Q^* \subset (\mul Q^{**})^\perp,
\]
which shows $Q\varphi=(I-P) Q\varphi$,
when $P$ is the orthogonal projection onto $\mul Q^{**}$.
Hence, one sees the inclusion
\begin{equation}\label{ququ}
\dot{Q} \subset Q_{\rm reg},
\end{equation}
which shows that $\dot{Q} \in \bL(\sH, \sK)$ is closable.
Observe that $\dot{Q}=Q_{\rm reg}$ if and only if $\cR=\dom Q$.
Now let $\{\varphi, \varphi'\}, \{\psi, \psi'\} \in S_\st$,
so that in particular $\{Q\varphi, \varphi'\} \in Q^*$ and $\{\psi, Q\psi\} \in Q$.
As a consequence, it follows by Definition \ref{ssDef}
that the form $\st(S_\st) \in \bF(\sH)$
is given by
\begin{equation}\label{sts01}
 \st(S_\st)[\varphi, \psi] =c(\varphi, \psi)+(\dot{Q}\varphi, \dot{Q} \psi),
 \quad \varphi, \psi \in \dom \dot{Q}=\dom S_\st=\cR,
\end{equation}
which confirms that $\st(S_\st)$ is closable.
It follows from the inclusion \eqref{ququ} that
\[
 \st(S_\st) \subset \sr,
\]
and, since both forms are closable, one obtains.
$ \bar \st(S_\st) \subset \bar \st_{\sr}$.
This last inclusion implies the inequality $\bar \st_{\sr} \leq \bar \st(S_\st)$
or, equivalently,
\[
 c+Q_{\rm reg}^*Q_{\rm reg}^{**}=(S_\sr)_{\rm F} \leq (S_\st)_{\rm F} =c+\dot{Q}^* \dot{Q}^{**},
\]
see \cite{HS2023b}. This shows the ordering between the Friedrichs extension of $S_\sr$
and the Friedrichs extension of $S_\st$.

\medskip

The above discussion will be illustrated by the following extreme example
involving a nonnegative and singular form.

\begin{proposition}\label{ccorr}
Let $\st \in \bF(\sH)$ be a nonnegative singular form
 and let $Q \in \bL(\sH, \sK)$ be a singular representing map for $\st$, so that
\begin{equation*}
 \st[\varphi, \psi]=(Q \varphi, Q \psi), \quad \varphi, \psi \in \dom \st=\dom Q,
\end{equation*}
in which case $Q_{\rm reg}=0$ and $\sr=\st_{\rm reg}=0$.
Then the following relations, corresponding
to the forms $\st$ and $\sr$ in \eqref{fortt0} and \eqref{fortt00}, are singular:
\begin{equation}\label{SArel1}
   \left\{ \begin{array}{l}
    S_{\sr} =(Q_{\rm reg})^*Q_{\rm reg}=\dom Q \times (\dom Q)^\perp, \\
    \wt A_\st=(Q_{\rm reg})^*(Q_{\rm reg})^{**}=\wt A_{\sr}=\cdom Q \times (\dom Q)^\perp,
\end{array}
\right.
\end{equation}
and $S_{\sr}$ is essentially selfadjoint. Furthermore, the following relations
corresponding to the form $\st(S_\st)$ in \eqref{sts01}, are singular,
\begin{equation}\label{SArel2}
\left\{ \begin{array}{l}
S_\st  =Q^*Q=\ker Q \times (\dom Q)^\perp, \\
(S_\st)_{\rm F}=Q^*Q^{**}=\cker Q \times (\ker Q)^\perp,
\end{array}
\right.
\end{equation}
and $S_\st$ is essentially selfadjoint if and only if $\cdom Q=\cker Q$.
\end{proposition}

\begin{proof}
Since $\st$ is singular one has $\sr=\st_{\rm reg}=0$.
Thus, equivalently,  $Q=Q_{\rm sing}$ or $Q_{\rm reg}=0$. Therefore
it is clear that
 \[
Q_{\rm reg}=\dom \st \times \{0\}, \quad
(Q_{\rm reg})^{*}=\sK \times (\dom \st)^\perp,
 \quad (Q_{\rm reg})^{**}= \cdom \st \times \{0\},
\]
from which the assertions in \eqref{SArel1} follow.

Since $Q$ is singular, recall that $\dom Q^*=\ker Q^*$.
 Let $\{f,f'\} \in Q^*Q$, so that  $\{f,\varphi\} \in Q$ and $\{\varphi,f'\} \in Q^*$
for some $\varphi \in \sK$.  In fact, $\varphi \in \ran Q \cap \ker Q^*=\{0\}$.
Thus $f \in \ker Q$ and $f'\in \mul Q^*$, which gives $\{f,f'\} \in \ker Q \times \mul Q^*$.
Therefore $Q^*Q \subset \ker Q \times \mul Q^*$. Since the reverse inclusion is obvious,
one obtains the identity $Q^*Q=\ker Q \times \mul Q^*$. Thus \eqref{SArel2} follows.
\end{proof}

As an application a Lebesgue decomposition of a form
involving nonzero regular and singular parts is given in the following example.

\begin{example}
Let $\st \in \bF(\sH)$ be the singular form from  Proposition \ref{ccorr}
and assume without loss of generality that the singular representing map
$Q \in \bL(\sH,\sK)$ for $\st$ is minimal; see \cite{HS2023b}).
 Since $\st$ is singular its lower bound satisfies $\gamma=0$, see \cite{HS2023b}.
Observe that $\st_{\rm reg}$ has representing map $Q_{\rm reg}=\dom \st \times \{0\}$,
while $\st_{\rm sing}$ has $Q$ as a representing map.
Now consider the shifted form $\st-c \in \bF(\sH)$ with $c < \gamma=0$.
Observe that
\[
 (\st-c)_{\rm reg}=\st_{\rm reg}-c =-c
  \quad \mbox{and} \quad  (\st-c)_{\rm sing}=\st_{\rm sing}=\st,
\]
see \cite[Corollary~2.3]{S3} or \cite{HS2023b}. This leads to
 the Lebesgue decomposition of the form $\st-c$:
\begin{equation}\label{A1}
 \st-c=(\st-c)_{\rm reg}+(\st-c)_{\rm sing}=-c + \st_{\rm sing}=-c + \st,
\end{equation}
see \cite{HSS2018,HS2023a,HS2023b}.
In order to determine a representing map for $\st-c$,
first define for $c< \gamma$
\begin{equation}\label{qc}
 q_c:\sH \to \sH =\cdom \st \oplus  (\dom \st)^\perp, \quad
 \varphi \to \begin{pmatrix} \sqrt{|c|} \varphi \\ 0 \end{pmatrix}, \quad \varphi \in \dom \st,
\end{equation}
so that $q_c$ is a representing map for the nonnegative form
$-c(\varphi,\psi)$, $\varphi, \psi \in \dom \st$.
It is clear that the column operator $Q_c$ defined by
\begin{equation}\label{SArell}
 Q_c=\begin{pmatrix}  q_c \\ Q \end{pmatrix}: \dom\st \to \sH \times \sK,
\end{equation}
is a representing map for $\st-c$, which in general is not minimal.
Consider the Lebesgue decomposition of $Q$: $Q=Q_{\rm reg}+Q_{\rm sing}$.
Since $Q$ is singular and $q_c$ is a bounded operator on $\dom \st$
one sees that
\[
\mul Q_c^{**}=\{0\} \times \mul Q^{**} = \{0\} \times \sK \subset \sH \times \sK,
\]
where the second identity holds by minimality of $Q$.
Hence, with $P_\sK$ the orthogonal projection from
$\sH \times \sK$ onto $\sK$, it follows that
\begin{equation}\label{A2}
Q_{c,{\rm reg}}=(I_{\sH\times\sK}-P_\sK)Q_c=q_c
 \quad \mbox{and} \quad Q_{c,{\rm sing}}=P_\sK Q_c=Q
\end{equation}
are representing maps for $(\st-c)_{\rm reg}$ and $(\st-c)_{\rm sing}$.

Next the interest is in the semibounded relations generated
by the form $\st-c$, given in terms of the
corresponding representing map $Q_c$.
It follows from the definition in \eqref{qc} that
 \begin{equation}\label{qc1}
 (Q_{c,{\rm reg}})^*=q_c^*
 =\big\{ \{\psi, \psi'\} \in \sH \times \sH :\,
 \psi' -\sqrt{|c|} \,\psi \in (\dom \st)^\perp \big\}.
\end{equation}
Therefore it is clear from \eqref{qc1} that
\begin{equation}\label{Qreg00-}
 (Q_{c,{\rm reg}})^* Q_{c,{\rm reg}}
 =q_c^*q_c=\big\{ \{\varphi, -c \varphi+\chi\} :\,
 \varphi \in \dom \st, \,\, \chi \in (\dom \st)^\perp \big\},
\end{equation}
and
\begin{equation}\label{Qreg00}
 (Q_{c,{\rm reg}})^* (Q_{c,{\rm reg}})^{**}
 =q_c^*q_c^{**}
 =\big\{ \{\varphi, -c \varphi+\chi\} :\, \varphi \in \cdom \st, \,\, \chi \in (\dom \st)^\perp \big\}.
\end{equation}
On the other hand, using \eqref{SArell} and Proposition~\ref{ccorr} one concludes that
\[
 Q_{c}^*Q_{c}=Q^*Q+q_c^*q_c=\big\{ \{\varphi, -c \varphi+\chi\} :\,
 \varphi \in \ker \st, \,\, \chi \in (\dom \st)^\perp \big\},
\]
and
\[
 Q_{c}^*Q_{c}^{**}=Q^*Q^{**}+q_c^*q_c^{**}=\big\{ \{\varphi, -c \varphi+\chi\} :\,
 \varphi \in \cker \st, \,\, \chi \in (\dom \st)^\perp \big\}.
\]
Thus one has
\[
Q_{c}^*Q_{c}=Q^*Q-c \quad \mbox{and} \quad Q_{c}^*Q_{c}^{**}=Q^*Q^{**}-c
\]
Therefore the relations $S_{\st-c}$ and $\wt A_{\st-c}$
associated with the shifted form $\st-c$ in \eqref{fortt1} are given by
\[
 S_{\st-c}=Q_{c}^*Q_{c}=Q^*Q-c, \quad  \wt A_{\st-c}=Q_{c}^*Q_{c}^{**}=Q^*Q^{**}-c,
\]
with the explicit formulas for $S_\st=Q^*Q$ and $\wt A_\st=Q^*Q^{**}$ as given in Proposition~\ref{ccorr}.
In particular, $(S_{\st-c})^{**}\subset  \wt A_{\st-c}$,
but $(S_{\st-c})^{**} \neq \wt A_{\st-c}$ whenever the form $\st$ is not a zero form on $\dom \st$.
Finally, from \eqref{SArel1}, \eqref{Qreg00-} and \eqref{Qreg00}, it is directly seen that
\[
 (Q_{c,{\rm reg}})^* (Q_{c,{\rm reg}})^{**}=q_c^*q_c^{**}
 =Q^*Q^{**}-c=\wt A_{\st-c}=(S_{\sr-c})^{**}=(S_\sr)^{**}-c.
\]   
The connection between Lebesgue decompositions of linear relations in \eqref{A2} and 
Lebesgue decompositions of semibounded forms in \eqref{A1}
is a general fact, which relies on the existence of representing maps for forms;
for details see \cite{HS2023seq, HS2023a, HS2023b}.
\end{example}

\end{document}